\documentclass[11pt,reqno]{amsart}
\usepackage{a4wide}

\numberwithin{equation}{section}
\usepackage{mathrsfs}
\usepackage{amsfonts}
\usepackage{amsmath}
\usepackage{stmaryrd}
\usepackage{amssymb}
\usepackage{amsthm}
\usepackage{mathrsfs}
\usepackage{url}
\usepackage{amsfonts}
\usepackage{amscd}
\usepackage{indentfirst}
\usepackage{enumerate}
\usepackage{amsmath,amsfonts,amssymb,amsthm}
\usepackage{amsmath,amssymb,amsthm,amscd}
\usepackage{graphicx,mathrsfs}
\usepackage{appendix}
\usepackage{color}
 \usepackage[colorlinks, linkcolor=blue, citecolor=blue]{hyperref}
\usepackage{exscale}
\usepackage{relsize}

\newcommand{\R}{\mathbb{R}}

\newcommand{\dis}{\displaystyle}

\renewcommand{\theequation}{\arabic{section}.\arabic{equation}}

\newtheorem{Thm}{Theorem}[section]
\newtheorem{Lem}[Thm]{Lemma}
\newtheorem{Cor}[Thm]{Corollary}
\newtheorem{Prop}[Thm]{Proposition}

\newtheorem{Rem}[Thm]{Remark}

\allowdisplaybreaks

\begin{document}

\title[Boundary Layers for the Lane-Emden Systems]{Existence of Boundary Layers for the supercritical Lane-Emden Systems}

\author{Qing Guo,\,\, Junyuan Liu, \,\, Shuangjie Peng}
 \address[Qing Guo]{College of Science, Minzu University of China, Beijing 100081, China} \email{guoqing0117@163.com}

 \address[Junyuan Liu]{ School of Mathematics and  Statistics, Central China Normal University, Wuhan, P.R. China}\email{jyliuccnu@163.com}

\address[Shuangjie Peng]{ School of Mathematics and  Statistics, Central China Normal University, Wuhan, P.R. China}\email{ sjpeng@ccnu.edu.cn}


\date{\today}

\begin{abstract}
We consider the following supercritical problem for the Lane-Emden system:
\begin{equation}\label{eq00}
\begin{cases}
-\Delta u_1=|u_2|^{p-1}u_2\  &in\ D,\\
-\Delta u_2=|u_1|^{q-1}u_1 \  &in\ D,\\
u_1=u_2=0\ &on\ \partial D,
\end{cases}
\end{equation}
where $D$ is a bounded smooth domain in $\R^N$, $N\geq4.$ What we mean by supercritical is that the exponent pair $(p,q)\in(1,\infty)\times(1,\infty)$ satisfies $\frac1{p+1}+\frac1{q+1}<\frac{N-2}N$. We prove that for some suitable domains $D\subset\R^N$, there exist positive solutions with  layers concentrating along one or several $k$-dimensional sub-manifolds of $\partial D$ as $$\frac1{p+1}+\frac1{q+1}\rightarrow\frac{n-2}{n},\ \ \ \ \frac{n-2}{n}<\frac1{p+1}+\frac1{q+1}<\frac{N-2}N,$$
 where $n:=N-k$ with $1\leq k\leq N-3$.

By transforming the original problem \eqref{eq00} into a lower $n$-dimensional weighted system, we carry out the reduction framework and apply the  blow-up analysis. The properties of the ground states related to the limit problem play a crucial role in this process. The corresponding exponent pair $(p_0,q_0)$, which represents the limit pair of $(p,q)$, lies on the critical hyperbola $\frac n{p_0+1}+\frac n{q_0+1}=n-2$. It is widely recognized that the range of the smaller exponent, say $p_0$, has a profound impact on the solutions, with $p_0=\frac n{n-2}$ being a threshold.

It is worth emphasizing that this paper tackles the problem by considering two different ranges of $p_0$, which is contained in $p_0>\frac n{n-2}$ and $p_0<\frac n{n-2}$ respectively. The coupling mechanisms associated with these ranges are completely distinct, necessitating different treatment approaches. This represents the main challenge overcome and the novel element of this study.
\end{abstract}

\maketitle
%\baselineskip 18p

\section{Introduction and main results}
\setcounter{equation}{0}
\subsection{Backgrounds}
The standard Lane-Emden system
%In this paper, we are concerned with  the following elliptic system of Hamiltonian type
\begin{equation}\label{eq1}
\begin{cases}
-\Delta u_1=|u_2|^{p-1}u_2\  &in\ D,\\
-\Delta u_2=|u_1|^{q-1}u_1 \  &in\ D,\\
u_1=u_2=0,\ &on\ \partial D,
\end{cases}
\end{equation}
%where $D$ is a bounded smooth domain in $\R^N$ with $N\geq4$.
\iffalse
We are aimed to show that %and $p,q\in(1,\infty)$ satisfying $\frac1{p+1}+\frac1{q+1}<\frac{N-2}N$.
for each set of positive integers $k_1,k_2,\ldots,k_m$ with $k_1+k_2+\cdots+k_m=k\leq N-3$ and for any $p,q\in(1,\infty)$ satisfying $\frac1{p+1}+\frac1{q+1}<\frac{N-2}N$ and
 $\frac1{p+1}+\frac1{q+1}\downarrow\frac{N-k-2}{N-k}$, problem
\eqref{eq1} possesses a positive solution, which concentrates along a $k$-dimensional sub-manifold of the boundary
$\partial D$, diffeomorphic to the product of spheres $\mathbb S^{k_1}\times \mathbb S^{k_2}\times\ldots\times \mathbb S^{k_m}$.
\fi
%The standard Lane-Emden system \eqref{eq1}
with a smooth bounded domain  $D\subset\R^N$ for $N\geq3$ and $p,q\in(0,\infty)$
is a typical Hamiltonian-type strongly coupled elliptic systems, which have been a subject of intense interest and has a rich structure.
Due to the fact that  tools  for analyzing a single equation cannot be used in a direct way to treat the systems, there have been very few results on the existence of solutions for strongly indefinite systems and their qualitative properties.
 One of the first result about positive solutions of \eqref{eq1} appeared in \cite{c-f-m} based on topological arguments. In \cite{f-f}, a variational argument relying on a linking theorem was used to show an existence result.
 %Then many efforts have been devoted to the variational study on such elliptic system leading to strongly indefinite functionals.
 In \cite{b-s-r}, the existence, positivity and uniqueness of ground state solutions for \eqref{eq1} was studied.  One may also refer to  \cite{s} and the surveys in \cite{f}.

 It is well known that the system is strongly affected by the values of the couple $(p,q)$.
 The existence theory is associated with the critical hyperbola
\begin{align}\label{pq}
\frac1{p+1}+\frac1{q+1}=\frac{N-2}N,
\end{align}
 which was introduced by  \cite{clement-figueiredo-mitidieri}
and \cite{vorst}.
We also have already known that in the critical or supercritical case, i.e.  $
\frac1{p+1}+\frac1{q+1}\leq\frac{N-2}N,$ if the domain $D$ is star-shaped, then \eqref{eq1} has no solutions.
According to \cite{figueiredo-felmer,hulshof-vorst}
and \cite{bonheure},
if $pq\neq1$ and  in the subcritical case
$
\frac1{p+1}+\frac1{q+1}>\frac{N-2}N,
$
then problem \eqref{eq1} has a solution.
 Moreover, Kim and Moon  \cite{kim-moon} considered the family of positive solutions of \eqref{eq1} on a smooth bounded convex domain $D$ in $\R^N$
 \iffalse
\begin{equation*}
\begin{cases}
-\Delta u_1=u_2^{p},\  &in\ D,\\
-\Delta u_2=u_1^{q_\epsilon}, \  &in\ D,\\
u_1,u_2>0,\ &in\ D,\\
u_1=u_2=0,\ &on\ \partial D,
\end{cases}
\end{equation*}
\fi
for $N\geq4$, $\max\{1,\frac3{N-2}\}<p<q$ with subcritical condition $
\frac1{p+1}+\frac1{q+1}=\frac{N-2+\epsilon}N>\frac{N-2}N$, and show that the multiple bubbling phenomena may arise with a detailed qualitative and quantitative description.

The first study by use of the Lyapunov reduction method about the bubbling solution was obtained by Kim and Pistoia
 %By employing the non-degeneracy result on the standard bubble due to \cite{frank-kim-pistoia}, Kim and Pistoia
 in \cite{KP}, where they  built multi-bubble solutions
to some critical problem, that is the
Brezis-Nirenberg type problem associated to \eqref{eq1}:
\begin{equation}\label{eqkp}
\begin{cases}
-\Delta u_1=|u_2|^{p-1}u_2+\epsilon(\alpha u_1+\beta_1u_2)\  &in\ D,\\
-\Delta u_2=|u_1|^{q-1}u_1+\epsilon(\beta_2 u_1+\alpha u_2) \  &in\ D,\\
u_1=u_2=0\ &on\ \partial D,
\end{cases}
\end{equation}
where $D$ is a smooth bounded domain in $\R^N$, $N\geq3$, $\epsilon>0$ is a small parameter, $\alpha, \beta_1, \beta_2$ are real numbers, and $p,q$ lie on
the critical hyperbola \eqref{pq}. Notice that they focus on the case $p\in(1,\frac{N-1}{N-2})$.
\iffalse
In fact, they also mentioned the almost critical case
\begin{equation}\label{eqkp2}
\begin{cases}
-\Delta u_1=u_2^{p-\alpha\epsilon},\  &in\ D,\\
-\Delta u_2=u_1^{q-\beta\epsilon}, \  &in\ D,\\
u_1,u_2>0,\ &in\ D,\\
u_1=u_2=0,\ &on\ \partial D,
\end{cases}
\end{equation}
with $\frac\alpha{(p+1)^2}+\frac\beta{(q+1)^2}>0$, and gave the existence of positive solutions which blow up at one or more points of $D$ as $\epsilon\rightarrow0$.

and as indicated by  \cite{kim-moon},  their method  can also cover the case $p\in[\frac{N-1}{N-2},\frac{N+2}{N-2})$.
\fi

\medskip
It is natural to believe that the system involving
 the supercritical condition $\frac1{p+1}+\frac1{q+1}<\frac{N-2}N$ would be more complex, and the existence of a nontrivial homology class in $D$ does not guarantee the existence of a nontrivial solution to \eqref{eq1}.  This can be seen  from the single Lane-Emden-Fowler problem
\begin{equation}\label{eqs}
-\Delta v=|v|^{p-1}v\ \ in\ D,\ \ v=0\ \ on\ \partial D.
\end{equation}
More precisely, for each integer $k$ such that $1\leq k\leq N-3$, Passaseo \cite{passaseo} found a bounded domain in $\R^N$, which is homotopically equivalent to the $k$-dimensional sphere, and proved that problem \eqref{eqs} does not have a nontrivial solution for $p+1\geq 2^*_{N,k}=\frac{2(N-k)}{N-k-2}$.  Clapp,  Faya and  Pistoia in \cite{c-f-p} gave
some examples of domains with richer homology, in which \eqref{eqs} does not have nontrivial solutions for $p>2^*_{N,k}-1$. On the other hand, for $p= 2^*_{N,k}-1$,
Wei and Yan \cite{wei-yan-jmpa} constructed infinitely many solutions of \eqref{eqs} in some domains. In \cite{ackermann-clapp-pistoia}, solutions of \eqref{eqs} concentrating at a $k$-dimensional sub-manifold for $p$ slightly below  $2^*_{N,k}$ were established.

\medskip
\subsection{Setting and assumptions}
In this present work, we are to investigate the supercritical problem of \eqref{eq1}.

We are aimed to show that %and $p,q\in(1,\infty)$ satisfying $\frac1{p+1}+\frac1{q+1}<\frac{N-2}N$.
for each set of positive integers $k_1,k_2,\ldots,k_m$ with $k_1+k_2+\cdots+k_m=k\leq N-3$ and for any $p,q\in(1,\infty)$ satisfying  $$\frac1{p+1}+\frac1{q+1}\rightarrow\frac{N-k-2}{N-k},\ \ \ \ \frac{N-k-2}{N-k}<\frac1{p+1}+\frac1{q+1}<\frac{N-2}N,$$
problem
\eqref{eq1} possesses a positive solution, which concentrates along a $k$-dimensional sub-manifold of the boundary
$\partial D$, diffeomorphic to the product of spheres $\mathbb S^{k_1}\times \mathbb S^{k_2}\times\ldots\times \mathbb S^{k_m}$.
\iffalse
for each set of positive integers $k_1,k_2,\ldots,k_m$ with $k_1+k_2+\ldots+k_m=k\leq N-3$, we prove that for some domain $D$,  problem
\eqref{eq1} possesses a positive solution for each supercritical couple $(p,q)$ satisfying
$
\frac1{p+1}+\frac1{q+1}<\frac{N-2}N$ and $\frac1{p+1}+\frac1{q+1}\downarrow\frac2{2^*_{N,k}}$. These solutions concentrate along a k-dimensional sub-manifold of the boundary
$\partial D$, diffeomorphic to the product of spheres $\mathbb S^{k_1}\times \mathbb S^{k_2}\times\ldots\times \mathbb S^{k_m}$. For this purpose,
\fi

For this purpose, we assume a bounded smooth domain $$\Omega\subset
\R^{n},\ \ with\ \ n=N-k$$ such that
\begin{align}
\overline\Omega\subset\{(x^1,\ldots,x^m,x')\in\R^m\times\R^{N-m-k}:x^i>0,i=1,\ldots,m\}.
\end{align}
Correspondingly,
\begin{align}
D:=\{(y^1,\ldots,y^m,z)\in\R^{k_1+1}\times\cdots\times\R^{k_m+1}\times\R^{N-m-k}:(|y^1|,\ldots,|y^m|,z)\in\Omega\}.
\end{align}
Then $D$ is a bounded smooth domain in $\R^N$ and invariant under the action of the group
 $\mathcal O:=\mathcal O(k_1+1)\times\ldots\times \mathcal O(k_m+1)$ on $\R^N$, where $\mathcal O(d)$ denotes
the group of all linear isometries of $\R^d$.

\medskip
In this paper, {\bf we assume that $(p_0,q_0)$, which is on the critical hyperbola:
\begin{align}\label{p0q0}\frac1{p_0+1}+\frac1{q_0+1}=\frac{n-2}{n},
\end{align}
satisfies that \begin{align}\label{P}
p_n<p_0<q_0\ \ with\ \ p_n=\max\Big\{1,\frac{3+\sqrt{4n+1}}{2(n-2)}\Big\}<\frac n{n-2}.
\end{align} }%, which means $p_0<\frac {n+2}{n-2}$ while $q_0>\frac {n+2}{n-2}$. 
\iffalse
 and $$p_0>\frac{n}{n-2},$$ or
\begin{align}\label{p0small}
\frac{n+2+\sqrt{(n+2)^2+8(n-2)^2}}{4(n-2)}<p_0<\frac n{n-2}.
\end{align}
\fi

Setting $$p=p_0-\alpha\epsilon,\ \ \ \ q=q_0-\beta\epsilon$$ with $\epsilon>0$, we search for $\mathcal O$-invariant solutions to \eqref{eq1} of the form
\begin{equation}\label{uv}
u_i(y^1,\ldots,y^m,z)=v_i(|y^1|,\ldots,|y^m|,z),\ \  i=1,2.
\end{equation}
Note that $(u_1,u_2)$ solves problem \eqref{eq1} if and only if $(v_1,v_1)$ solves
\begin{align}\label{eqv}
\begin{cases}
-div(a(x)\nabla v_1)=a(x)|v_2|^{p-1}v_2\ \ &in\ \Omega,\\
-div(a(x)\nabla v_2)=a(x)|v_1|^{q-1}v_1\ \ &in\ \Omega,\\
v_1=v_2=0\ \ &on\ \partial\Omega,
\end{cases}
\end{align}
where $a(x)=a(x^1,\ldots,x^{n})=(x^1)^{k_1}(x^2)^{k_2}\cdots (x^m)^{k_m}\in C^2(\overline\Omega)$ is strictly positive on $\overline\Omega$.

Then we are equivalently to construct solutions $(v_{1,\epsilon},v_{2,\epsilon})$ of \eqref{eqv}, which concentrates at some points $\xi_1,\ldots,\xi_\kappa\in\partial\Omega$
as $\epsilon\rightarrow0$. Correspondingly, by \eqref{uv}, there exists a solution $(u_{1,\epsilon},u_{2,\epsilon})$ of \eqref{eq1} with positive layers accumulating
along the $k$-dimensional sub-manifolds
\begin{align*}
M_j=\{(y^1,\ldots,y^m,z)\in\R^{k_1+1}\times\cdots\times\R^{k_m+1}\times\R^{N-m-k}:(|y^1|,\ldots,|y^m|,z)=\xi_j\}
\end{align*}
of the boundary $\partial D$ as $\epsilon\rightarrow0$, which is diffeomorphic to $\mathbb S^{k_1}\times\cdots\times \mathbb S^{k_m}$ where
$\mathbb S^d$ is the unit sphere in $\R^{d+1}$.

\medskip

More generally, we are to study  problem \eqref{eqv} with some  potential function $a\in C^2(\overline\Omega)$ which is strictly positive on $\overline\Omega$ and satisfies some more general conditions:\\
{\bf (a)} There exist $\kappa$ non-degenerate critical points $\tilde\xi_1,\ldots,\tilde\xi_\kappa\in\partial\Omega$ of the restriction of $a$ to
$\partial \Omega$ such that $$\langle\nabla a(\tilde\xi_i),\nu(\tilde\xi_i)\rangle>0,\ \ \forall i=1,\ldots,\kappa,$$
where $\nu(\tilde\xi_i)$ is the inward pointing unit normal to $\partial\Omega$ at $\tilde\xi_i$.

\begin{figure}
%\small
%\centering
\includegraphics[width=9cm]{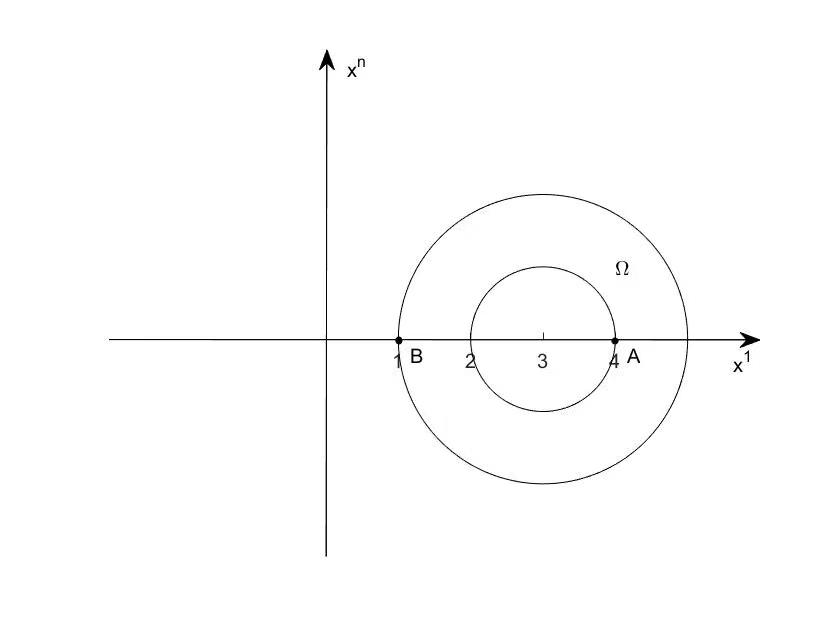}\\
\caption{$\Omega=\{(x^1,x^2,\ldots,x^n):1<(x^1-3)^2+(x^2)^2+\ldots+(x^n)^2<4\}$, $a(x)=(x^1)^k$, $\nu(A)=\nu(B)=(1,0,\ldots,0)$.
}
\label{1}
\end{figure}
\begin{figure}
%\small
%\centering
\includegraphics[width=9cm]{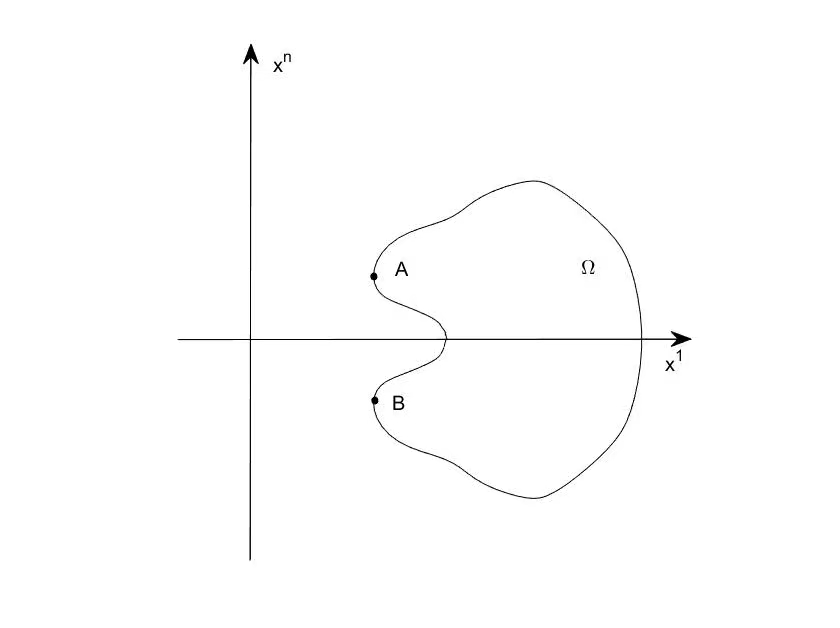}\\
\caption{ $\Omega$ is a bounded domain.}
\label{1}
\end{figure}
For instance, the domains $\Omega$ and points $A,B$ in Figure 1 and Figure 2 satisfy {\bf (a)}.

\medskip
\subsection{Main results}

Before giving our main theorem, we first briefly introduce the limit problem, leaving more details in Section 2.
A positive ground state $(U,V)$ to the following system was found in \cite{lions},
\begin{align}\label{eqUV}
\begin{cases}
&\displaystyle-\Delta U=|V|^{p_0-1}V,\ \ in\ \ \mathbb R^n,\\
&\displaystyle-\Delta V=|U|^{q_0-1}U,\ \ in\ \ \mathbb R^n,\\
&\displaystyle(U,V)\in \dot{W}^{2,\frac{p_0+1}{p_0}}(\mathbb R^n)\times\dot{W}^{2,\frac{q_0+1}{q_0}}(\mathbb R^n),
\end{cases}
\end{align}
where $n\geq3$ and $(p_0,q_0)$ satisfy \eqref{p0q0}.
By Sobolev embeddings, there holds that
\begin{align}\label{emb}
\begin{cases}
\displaystyle \dot{W}^{2,\frac{p_0+1}{p_0}}(\mathbb R^n)\hookrightarrow \dot{W}^{1,p^*}(\mathbb R^n)\hookrightarrow L^{q_0+1}(\R^n),\\
\displaystyle  \dot{W}^{2,\frac{q_0+1}{q_0}}(\mathbb R^n)\hookrightarrow \dot{W}^{1,q^*}(\mathbb R^n)\hookrightarrow L^{p_0+1}(\R^n),
\end{cases}
\end{align}
with
$$\frac1{p^*}=\frac {p_0}{p_0+1}-\frac1n=\frac1{q_0+1}+\frac1n,\ \ \ \ \frac1{q^*}=\frac {q_0}{q_0+1}-\frac1n=\frac1{p_0+1}+\frac1n,$$
and so the following energy functional is well-defined in $\dot W^{2,\frac{p_0+1}{p_0}}(\R^n)\times\dot W^{2,\frac{q_0+1}{q_0}}(\R^n)$:
\begin{align*}
\tilde I_0(u,v):=\int_{\R^n}\nabla u\cdot\nabla v
-\frac1{p_0+1}\int_{\R^n}| v|^{p_0+1}-\frac1{q_0+1}\int_{\R^n}|u|^{q_0+1}.
\end{align*}
According to \cite{alvino-lions-trombetti},  the ground state is radially symmetric and decreasing up to a suitable translation.
Thanks to  \cite{hulshof-vorst} and \cite{wang}, the positive ground state $(U_{0,1},V_{0,1})$ of \eqref{eqUV} is unique with $U_{0,1}(0)=1$
and the family of functions
\begin{align*}
(U_{\xi,\lambda}(y),V_{\xi,\lambda}(y))=(\lambda^{\frac n{q+1}}U_{0,1}(\lambda(y-\xi)),\lambda^{\frac n{p+1}}V_{0,1}(\lambda(y-\xi)))
\end{align*}
for any $\lambda>0,\xi\in\mathbb R^n$ also solves system \eqref{eqUV}.
Sharp asymptotic behavior  of the ground states to \eqref{eqUV}  (see \cite{hulshof-vorst}) and the non-degeneracy for \eqref{eqUV}  at each ground state (see \cite{frank-kim-pistoia}) play an important role to construct bubbling solutions especially using the Lyapunov-Schmidt reduction methods.

\medskip

Our main results in this paper can be stated as follows.
\begin{Thm}\label{th1}
Given $1\leq k\leq N-3$, there exists $\epsilon_0>0$ such that, for each $\epsilon\in(0,\epsilon_0)$, problem \eqref{eqv} has a solution $(v_{1,\epsilon},v_{2,\epsilon})$
of the form
\begin{align}\label{constructv}
v_{1,\epsilon}=\sum_{i=1}^\kappa U_i+o(1),\ \ \
v_{2,\epsilon}=\sum_{i=1}^\kappa V_i+o(1)
\end{align}
where $U_i=U_{\xi_{i,\epsilon},\delta_{i,\epsilon}},V_i=V_{\xi_{i,\epsilon},\delta_{i,\epsilon}}$ and
 $\epsilon^{-\frac{n-1}{n-2}}\delta_{i,\epsilon}\rightarrow \Lambda_i>0$, $\xi_{i,\epsilon}\rightarrow\tilde\xi_i\in\partial\Omega$ for $i=1,\ldots,\kappa$ as $\epsilon\rightarrow0$.

\end{Thm}

\begin{Rem}
Note that in condition \eqref{P}, $p_n=1$ when $n\geq6$, which indicates that {\bf in the case of $n\geq6$ we in fact span the entire range of $p_0\in(1,\frac{n+2}{n-2})$}.
\smallskip

The condition \eqref{P} covers three distinct ranges for the parameter $p_0$:  $p_n<p<\frac n{n-2}, p=\frac n{n-2}$ and $\frac n{n-2}<p<\frac{n+2}{n-2}$. Since the case when $p_0=\frac n{n-2}$ can be treated by slightly modifying the proof of that when $p_0>\frac n{n-2}$ (in view of Lemma \ref{lemasym}), so we omit the details to maintain focus.

 \medskip
 
 The coupling mechanism of the strongly indefinite problem in these two cases is totally different.  Even in the case of $p_0>\frac{n}{n-2}$, the blow-up scenario is not the same as that of the single Lane-Emden equation, and the standard approach does not work well, which forces us to adopt some new approach and analysis.

\medskip

It is worth noting that  when $p_0<\frac n{n-2}$, the system \eqref{eq1} exhibits   stronger nonlinear feature that the single equation does not have. 
%This can be seen from section 2 below that
The essential reason  lies in the fact that the decay order at infinity of $U$ is strictly smaller than that of the fundamental  solution of $-\Delta$ in $\R^n$ (see Lemma \ref{lemasym}), resulting in significant differences in the properties of the solution compared to the single equations.
\medskip

Unlike  $p_0>\frac n{n-2}$, when $p_0<\frac n{n-2}$, the approximate solution of the Lane-Emden system cannot simply be expressed using the ground state solution of the limiting problem and the regular part $H$ of the Green's function. More precisely, in this case, the characterization of the boundary behavior of the harmonic function $h_i=U_i-PU_i$ in section 2.2 becomes rougher and more difficult to control than that of $H$.

\end{Rem}

\begin{Rem}%%%%%%%%%%%%%%%%%%%%%%%%%%%%%% 变号解
The existence of sign-changing solutions can be derived from the proof of Theorem \ref{th1}.

There exists some $\epsilon_0>0$ such that, for each $\lambda_1,\ldots,\lambda_\kappa\in\{0,1\}$ and $\epsilon\in(0,\epsilon_0)$, problem \eqref{eqv} has a solution $(v_{1,\epsilon},v_{2,\epsilon})$
of the form
\begin{align}\label{constructv}
v_{1,\epsilon}=\sum_{i=1}^\kappa (-1)^{\lambda_i}U_i+o(1),\ \ \
v_{2,\epsilon}=\sum_{i=1}^\kappa  (-1)^{\lambda_i}V_i+o(1)
\end{align}
where $U_i=U_{\xi_{i,\epsilon},\delta_{i,\epsilon}},V_i=V_{\xi_{i,\epsilon},\delta_{i,\epsilon}}$ and
 $\epsilon^{-\frac{n-1}{n-2}}\delta_{i,\epsilon}\rightarrow \Lambda_i>0$, $\xi_{i,\epsilon}\rightarrow\tilde\xi_i\in\partial\Omega$ for $i=1,\ldots,\kappa$ as $\epsilon\rightarrow0$.

\end{Rem}

\medskip

Finally, Theorem \ref{th1} implies the following results back to the original problem \eqref{eq1}.

Precisely, given $k_1,\ldots,k_m\in\mathbb N$ with $k=k_1+\ldots+k_m\leq N-3$, for $\xi\in\R^{N-k},\delta>0$, we set $$\widetilde U_{\xi,\delta}(y^1,y^2,\ldots,y^m,z)=U_{\xi,\delta}(|y^1|,|y^2|,\ldots,|y^m|,z).$$
\begin{Thm}
There exists $\epsilon_0>0$ such that problem \eqref{eq1} has a solution $(u_{1,\epsilon},u_{2,\epsilon})\in W^{1,p^*}_0(D)\times W^{1,q^*}_0(D)$ of the form
\begin{align}\label{constructu}
u_{1,\epsilon}=\sum_{i=1}^\kappa \widetilde U_{\xi_{i,\epsilon},\delta_{i,\epsilon}}+o(1),\ \ \
u_{2,\epsilon}=\sum_{i=1}^\kappa \widetilde V_{\xi_{i,\epsilon},\delta_{i,\epsilon}}+o(1)
\end{align}
with
 $\epsilon^{-\frac{n-1}{n-2}}\delta_{i,\epsilon}\rightarrow \Lambda_i>0$, $\xi_{i,\epsilon}\rightarrow\tilde\xi_i\in\partial\Omega$ for $i=1,\ldots,\kappa$ as $\epsilon\rightarrow0$.

\end{Thm}
\medskip

This paper is organized as follows. In section 2, we study the projection of the bubbles in two different cases: $p_0\in(\frac n{n-2},\frac{n+2}{n-2})$ and $p_0\in(p_n,\frac n{n-2})$ respectively.
The problem setting and function space $X_{p,q}$ are introduced in section 3,
where we give an equivalent form to \eqref{eq1} to carry out the reduction framework and define the approximate solutions.
In section 4, we perform the linear analysis and solve the auxiliary nonlinear problem, reducing the problem to finding a critical point of some function $J_\epsilon$,
which is called the reduced energy on a finite-dimensional set $\Gamma$.
Some basic estimates on the reduced energy are put in the appendix.

\medskip

\section{Projection of the bubbles}

Recall that the bubbles satisfy the following properties.

\begin{Lem}\label{lemasym}\cite{hulshof-vorst}
Assume that $p_0\leq\frac{n+2}{n-2}$. There exist some positive constants $a=a_{n,p_0}$ and $b=b_{n,p_0}$ depending only on $n$ and $p_0$ such that
\begin{align}\label{asymV}
&\lim_{r\rightarrow\infty}r^{n-2}V_{0,1}(r)=b_{n,p_0};
\end{align}
while
\begin{align}\label{asymU}
\begin{cases}
\dis \lim_{r\rightarrow\infty}r^{n-2}U_{0,1}(r)=a_{n,p_0}\ \ \ \ \ \ \ &\text{if}\  p_0>\frac n{n-2};\vspace{0.12cm}\\
\dis \lim_{r\rightarrow\infty}\dis\frac{r^{n-2}}{\log r}U_{0,1}(r)=a_{n,p_0}\ \ \ \ \ \ \ &\text{if}\ p_0=\frac n{n-2};\vspace{0.12cm}\\
\dis \lim_{r\rightarrow\infty}r^{(n-2)p_0-2}U_{0,1}(r)=a_{n,p_0}\ \ &\text{if}\ p_0<\frac n{n-2}.
\end{cases}
\end{align}
Furthermore, in the last case, we have $b_{n,p_0}^{p_0}=a_{n,p_0}((n-2)p_0-2)(n-(n-2)p_0)$.
\end{Lem}

\begin{Lem}\label{lemnonde}\cite{frank-kim-pistoia}
Set
$$
(\Psi_{0,1}^0,\Phi_{0,1}^0)=\Big(y\cdot\nabla U_{0,1}+\frac{n U_{0,1}}{q_0+1},y\cdot\nabla V_{0,1}+\frac{nV_{0,1}}{p_0+1}\Big)
$$
 and
$$(\Psi_{0,1}^l,\Phi_{0,1}^l)=(\partial_l U_{0,1},\partial_l V_{0,1}),\ \ for\ \ l=1,\ldots,n.$$
Then the space of solutions to the linear system
\begin{align}
\begin{cases}
&\displaystyle -\Delta\Psi=p_0V_{0,1}^{p_0-1}\Phi\ \ \text{in}\ \mathbb R^n,\vspace{0.12cm}\\
&\displaystyle -\Delta\Phi=q_0U_{0,1}^{q_0-1}\Psi\ \ \text{in}\ \mathbb R^n,\vspace{0.12cm}\\
&\displaystyle (\Psi,\Phi)\in  \dot{W}^{2,\frac{p_0+1}{p_0}}(\mathbb R^n)\times\dot{W}^{2,\frac{q_0+1}{q_0}}(\mathbb R^n)
\end{cases}
\end{align}
is spanned by
$$\Big\{(\Psi_{0,1}^0,\Phi_{0,1}^0),(\Psi_{0,1}^1,\Phi_{0,1}^1),\ldots,(\Psi_{0,1}^n,\Phi_{0,1}^n)\Big\}.$$
\end{Lem}

\medskip

Consider the solution of the form
\eqref{constructv}. Given $\kappa\in\mathbb N$ and for $i=1,\ldots,\kappa$,  we set
\begin{align}\label{set}
\begin{split}
&\delta_{i,\epsilon}=\begin{cases}\dis\epsilon^{\frac{n-1}{n-2}}\Lambda_i\ \ \ &if\  p_0>\frac n{n-2}\\
\dis\epsilon^{\frac{(n-2)p_0-1}{(n-2)p_0-2}}\Lambda_i\ \ \ &if\  p_0<\frac n{n-2}
\end{cases},\ \ \ \Lambda_i>0,\\
&\xi_{i,\epsilon}=\xi_i+\eta_i\nu(\xi_i),\ \ \ \xi_i\in\partial\Omega,\ \ \ \eta_i=\epsilon t_i.
\end{split}\end{align}
For simplicity, we denote
$$\vec{\xi}=(\xi_1,\ldots,\xi_\kappa)\in(\partial\Omega)^\kappa,\ \ \vec\Lambda=(\Lambda_1,\ldots,\Lambda_\kappa)\in(0,+\infty)^\kappa,\ \ \vec t=(t_1,\ldots,t_\kappa)\in(0,+\infty)^\kappa$$ and introduce a configuration space $\Gamma$ as the set of the concentration points and the concentration parameters as follows:
$$\Gamma:=\{(\vec\xi,\vec\Lambda,\vec t)\in(\partial\Omega)^\kappa\times(0,+\infty)^\kappa\times(0,+\infty)^\kappa:\xi_i\neq\xi_j\ if\ i\neq j,\ i,j=1,\dots,\kappa\}.$$

Given $(\vec\xi,\vec\Lambda,\vec t)\in\Gamma$, denote $$U_i=U_{\xi_{i,\epsilon},\delta_{i,\epsilon}},\  \ V_i=V_{\xi_{i,\epsilon},\delta_{i,\epsilon}},\ i=1,\ldots,\kappa.$$
For $i=1,\ldots,\kappa$, let $(PU_i,PV_i)$ be the unique smooth solution of the system
\begin{align}
\begin{cases}
-\Delta PU_i=V_i^{p_0}\ \ &in\ \Omega\\
-\Delta PV_i=U_i^{q_0}\ \ &in\ \Omega\\
PU_i=PV_i=0\ \ &on\ \partial\Omega.
\end{cases}
\end{align}
\medskip

\subsection{Projection for $p_0\in(\frac{n}{n-2},\frac{n+2}{n-2})$}
First recall the properties of the Green's function and its regular part.

Let $G=G_\Omega$ be the Green's function of the Laplacian $-\Delta$ in $\Omega$ with respect to the Dirichlet boundary condition, and $H=H_\Omega: \Omega\times
\Omega\rightarrow\mathbb R$ be its regular part.  For each $y\in\Omega$,
\begin{align*}
\begin{cases}
-\Delta_xH(x,y)=0 &x\in\Omega,\\
H(x,y)=\frac{\gamma_n}{|x-y|^{n-2}} &x\in\partial\Omega,
\end{cases}
\end{align*}where $\gamma_n:=\frac1{(n-2)|\mathbb S^{n-1}|}$.
 Then, $0<G(x,y)=G(y,x)=\frac{\gamma_n}{|x-y|^{n-2}}-H(x,y)<\frac{\gamma_n}{|x-y|^{n-2}}$  for $(x,y)\in\Omega\times\Omega,x\neq y$.

 We also need a precise behavior of $H(x,y)$ when $x$ and $y$ are close to the boundary. For this purpose, given $\eta>0$, let $\Omega_\eta:=\{x\in\Omega:dist(x,\partial\Omega)\leq\eta\}$. When $\eta$ is small enough, the orthogonal projection $p:\Omega_{2\eta}\rightarrow\partial\Omega$
 onto the boundary is well defined. For any $x\in\Omega_{2\eta}$ there exists a unique point $p(x)\in\partial\Omega$ with $d(x):=dist(x,\partial\Omega)=|p(x)-x|$.
 Let $\nu(x)$ denote the inward normal to $\partial\Omega$ at $x$. For $x\in\Omega_{2\eta}$, we define $\bar x:=p(x)-d(x)\nu(x)=x-2d(x)\nu(x)$, which is the reflection of $x$
 on $\partial\Omega$. The following known results are obtained in \cite{ackermann-clapp-pistoia}.
 \begin{Lem}\label{lemH}\cite{ackermann-clapp-pistoia}
 There exists $C>0$ such that for all $x\in\Omega_\eta$ and $y\in\Omega$, there hold that
 \begin{align*}
 &\Big|H(x,y)-\frac{\gamma_n}{|\bar x-y|^{n-2}}\Big|\leq\frac{Cd(x)}{|\bar x-y|^{n-2}},\\
 &\Big|\nabla_x\Big(H(x,y)-\frac{\gamma_n}{|\bar x-y|^{n-2}}\Big)\Big|\leq\frac{C}{|\bar x-y|^{n-2}}.\\
 %&\Big|\nabla_yH(x,y)+\frac{\gamma_n(n-2)(y-\bar x)}{|y-\bar x|^{n}}\Big|\leq\frac{Cd(x)}{d(y)|y-\bar x|^{n-2}},\\
 %&\nu_x\cdot\nabla_x H(x,y)|_{y=x}\geq Cd(x)^{1-n},\\
 %&\Big|\nabla_x\nabla_yH(x,y)+\gamma_n(n-2)\nabla_x\Big(\frac{y-\bar x}{|y-\bar x|^{n}}\Big)\Big|\leq C\Big(\frac{1}{d(y)|y-\bar x|^{N-2}}+\frac{1}{|y-\bar x|^{n-1}}\Big).
 \end{align*}
 In particular,
  \begin{align*}
 &0\leq H(x,y)\leq\frac{C}{|\bar x-y|^{n-2}},\ \ x\in\Omega_\eta,\ y\in\Omega,\ \ \
 \Big|\nabla_xH(x,y)\Big|\leq\frac{C}{|x-y|^{n-1}},\ \ x,y\in\Omega.
 \end{align*}

 \end{Lem}
\medskip

 A standard comparison argument based on Lemma \ref{lemasym} yields that
 \begin{Lem}\label{lemP}
 There exists $c>0$ such that for all $x\in\Omega$,
  \begin{align}\label{0}
  \begin{split}
 &0\leq PU_{\xi,\delta}\leq U_{\xi,\delta},\ \ \ 0\leq PV_{\xi,\delta}\leq V_{\xi,\delta},\\
 &0\leq U_{\xi,\delta}- PU_{\xi,\delta}\leq \frac{a_{n,p_0}}{\gamma_n}\delta^{\frac{n}{p_0+1}}H(x,\xi)\leq\frac{c_1\delta^{\frac{n}{p_0+1}}}{|x-\bar\xi|^{n-2}},\\
  &0\leq V_{\xi,\delta}- PV_{\xi,\delta}\leq\frac{b_{n,p_0}}{\gamma_n}\delta^{\frac{n}{q_0+1}}H(x,\xi)\leq\frac{c_2\delta^{\frac{n}{q_0+1}}}{|x-\bar\xi|^{n-2}}.
 \end{split}\end{align}

 Moreover, there hold that
 \begin{align*}&R_{1,\xi,\delta}(x):=PU_{\xi,\delta}-U_{\xi,\delta}+\frac{a_{n,p_0}}{\gamma_n}\delta^{\frac{n}{p_0+1}}H(x,\xi),\\
 &R_{2,\xi,\delta}(x):=PV_{\xi,\delta}-V_{\xi,\delta}+\frac{b_{n,p_0}}{\gamma_n}\delta^{\frac{n}{q_0+1}}H(x,\xi)
 \end{align*}
 satisfies  \begin{align}\label{estR}
  \begin{split}\|R_{1,\xi,\delta}\|_{L^\infty(\Omega)}=O(\frac{\delta^{\frac{n}{p_0+1}+1}}{d(\xi)^{n-1}}),\ \ \
 \|R_{2,\xi,\delta}\|_{L^\infty(\Omega)}=O(\frac{\delta^{\frac{n}{q_0+1}+1}}{d(\xi)^{n-1}}).\end{split}\end{align}
 \end{Lem}

\begin{proof}
By use of the maximum principle and Lemma \ref{lemH}, it is suffices to show that
\begin{align}\label{Uasym}
  \begin{split}
  \left| U(x)-\frac{a_{n,p_0}}{|x|^{n-2}} \right|=O\Big(\frac1{|x|^{n-1}}\Big),\ \ \
  \left| V(x)-\frac{b_{n,p_0}}{|x|^{n-2}} \right|=O\Big(\frac1{|x|^{n-1}}\Big),
   \end{split}
\end{align}
which can be found in \cite{kim-moon}.

\end{proof}

\medskip

 \begin{Lem}\label{lemP'}
There exists $\sigma>0$ such that
\begin{align}\label{1}
\int_\Omega|\nabla PU_i|PV_j=O\left(\epsilon^{1+\sigma}\right),\ \ \int_\Omega|\nabla PV_i|PU_j=O\left(\epsilon^{1+\sigma}\right),
\end{align}
\begin{align}\label{2}
\|\nabla PU_i\|_{L^{\frac{p_0+1}{p_0}}(\Omega)}=O\left(\epsilon^{1+\sigma}\right),\ \
\|\nabla PV_i\|_{L^{\frac{q_0+1}{q_0}}(\Omega)}=O\left(\epsilon^{1+\sigma}\right).
\end{align}

 \end{Lem}

 \begin{proof}

From the integral equation of $PU_i$ we know that
\begin{align*}
\nabla PU_i(x)=\int_\Omega\nabla_x\Big(\frac{\gamma_n}{|x-y|^{n-2}}-H(x,y)
\Big)V_i^{p_0}dy.
\end{align*}

We first estimate \eqref{2}.  By use of Hardy-Littlewood-Sobolev inequalities, we have
\begin{align*}
\|\nabla PU_i\|_{L^{\frac{p_0+1}{p_0}}(\Omega)}
&\leq C\Big\|\int\frac1{|x-y|^{n-1}}V_i^{p_0}(y)dy\Big\|_{L^{\frac{p_0+1}{p_0}}(\Omega)}\\
&\leq C\|V_i^{p_0}\|_{L^{r}}
=O\left( \delta_{i,\epsilon}^{\frac n{r}-\frac{p_0}{p_0+1}n}\right)
=O\left(\delta_{i,\epsilon}^{\frac{n-2}{n-1}+\sigma}\right),
\end{align*}
where $1+\frac{p_0}{p_0+1}=\frac{n-1}{n}+\frac1{r}$ and we find $r>\frac{n}{p_0(n-2)}$.

Similarly,
\begin{align*}
&\|\nabla PV_i\|_{L^{\frac{q_0+1}{q_0}}(\Omega)}
=O\left(\delta_{i,\epsilon}^{\frac{n-2}{n-1}+\sigma}\right),
\end{align*}
and we have proved \eqref{2}.

Next, for \eqref{1},
\begin{align}\label{1'}
\begin{split}
\int_\Omega|\nabla PU_i|PV_j\leq C\|\nabla PU_i\|_{L^{\frac{p_0+1}{p_0}}(\Omega)}\|PV_j\|_{L^{p_0+1}}=O(\epsilon^{1+\sigma}).
\end{split}\end{align}

 Similarly,
 \begin{align}\label{1''}
\begin{split}
&\int_\Omega|\nabla PV_i|PU_j=O(\epsilon^{1+\sigma}).
\end{split}\end{align}

 \end{proof}
\medskip

\medskip

For $j=1,\ldots,n$, we also denote
\begin{align}\label{psiphi}
\Psi_{\xi,\delta}^0=\frac{\partial U_{\xi,\delta}}{\partial\delta},\ \Psi_{\xi,\delta}^j=\frac{\partial U_{\xi,\delta}}{\partial \xi^j}, \ \  \ \
\Phi_{\xi,\delta}^0=\frac{\partial V_{\xi,\delta}}{\partial\delta},\ \Phi_{\xi,\delta}^j=\frac{\partial V_{\xi,\delta}}{\partial \xi^j}.
\end{align}
It is known that the space spanned by the $n+1$ pairs $(\Psi_{\xi,\delta}^j,\Phi_{\xi,\delta}^j)$ is the set of solutions of the linearized problem
\begin{align}\label{eqlinear}
\begin{cases}
 \displaystyle -\Delta\Psi=p_0V_{\xi,\delta}^{p_0-1}\Phi \ \ &\text{in}\ \mathbb R^n,\vspace{0.12cm}\\
 \displaystyle -\Delta\Phi=q_0U_{\xi,\delta}^{q_0-1}\Psi \ \ &\text{in}\ \mathbb R^n,\vspace{0.12cm}\\
\displaystyle (\Psi,\Phi)\in  \dot{W}^{2,\frac{p_0+1}{p_0}}(\mathbb R^n)\times\dot{W}^{2,\frac{q_0+1}{q_0}}(\mathbb R^n).
\end{cases}
\end{align}

Denote $(\Psi_{i}^l,\Phi_{i}^l)=(\Psi_{\xi_{i,\epsilon},\delta_{i,\epsilon}}^l,\Phi_{\xi_{i,\epsilon},\delta_{i,\epsilon}}^l)$ for simplicity.
For $i=1,\ldots,\kappa, l=0,1,\ldots,n$, let the pair $(P\Psi_{i}^l,P\Phi_{i}^l)$ be the unique smooth solution of the system
\begin{align}\label{eqP}
\begin{cases}
 \displaystyle -\Delta P\Psi_{i}^l=p_0V_{i}^{p_0-1}\Phi_{i}^l \ \ &\text{in}\ \Omega,\vspace{0.12cm}\\
 \displaystyle -\Delta P\Phi_{i}^l=q_0U_{i}^{q_0-1}\Psi_{i}^l \ \ &\text{in}\ \Omega,\vspace{0.12cm}\\
\displaystyle P\Psi_{i}^l=P\Phi_{i}^l=0 \ \ &\text{on}\ \partial\Omega.
\end{cases}
\end{align}
Then by use of the comparison argument we have
\begin{Lem}\label{lemP2}
For $i=1,\ldots,\kappa$ and $l=0,1,\ldots,n$, for $x\in\Omega$,
  \begin{align*}
 & P\Psi_{i}^l=\begin{cases}
 \Psi_{i}^l+\frac{a_{n,p_0}}{\gamma_n}\delta_{i,\epsilon}^{\frac{n}{p_0+1}-1}H(x,\xi_{i,\epsilon})+o(\delta_{i,\epsilon}^{\frac{n}{p_0+1}-1}\eta_{i,\epsilon}^{-(n-2)}),\ &l=0\\
  \Psi_{i}^l+\frac{a_{n,p_0}}{\gamma_n}\delta_{i,\epsilon}^{\frac{n}{p_0+1}}\partial_{\xi,l}H(x,\xi_{i,\epsilon})+o(\delta_{i,\epsilon}^{\frac{n}{p_0+1}}\eta_{i,\epsilon}^{-(n-1)}),\ &l=1,\dots,n,
 \end{cases}
 \end{align*}
and
  \begin{align*}
 & P\Phi_{i}^l=\begin{cases}
 \Phi_{i}^l+\frac{b_{n,q_0}}{\gamma_n}\delta_{i,\epsilon}^{\frac{n}{q_0+1}-1}H(x,\xi_{i,\epsilon})+o(\delta_{i,\epsilon}^{\frac{n}{q_0+1}-1}\eta_{i,\epsilon}^{-(n-2)}),\ &l=0\\
  \Psi_{i}^l+\frac{b_{n,q_0}}{\gamma_n}\delta_{i,\epsilon}^{\frac{n}{q_0+1}}\partial_{\xi,l}H(x,\xi_{i,\epsilon})+o(\delta_{i,\epsilon}^{\frac{n}{q_0+1}}\eta_{i,\epsilon}^{-(n-1)}),\ &l=1,\dots,n,
 \end{cases}
 \end{align*}
where $\partial_{\xi,l}H(x,\xi_{i,\epsilon})$ is the $l-$th components of  $\nabla_{\xi_{i,\epsilon}} H(x,\xi_{i,\epsilon})$.

 \end{Lem}
\medskip

\subsection{Projection for $p_0\in(1,\frac{n}{n-2})$}
\medskip

For any $i\in\{1,\ldots,\kappa\}$, the harmonic function $h_i=U_i-PU_i$
satisfies that
\begin{align*}
\begin{cases}-\Delta h_i=0  &x\in\Omega,\\
h_i(x)=U_i(x) &x\in\partial\Omega.
\end{cases}
\end{align*}
On the one hand, the maximal principle implies that
\begin{align}\label{hi}
0\leq h_i(x)\leq\sup_{x\in\partial\Omega}U_i(x)\leq C\delta_{i,\epsilon}^{\frac{p_0n}{q_0+1}}\eta_{i,\epsilon}^{-\frac{p_0+1}{q_0+1}n}.
 \end{align}
 On the other hand, since $G(x,y)\leq\frac C{|x-y|^{n-2}}$, it can be proved that
  \begin{align*}
PU_i(x)\leq \int_\Omega G(x,y)V_i^{p_0}(y)dy\leq
\frac{C\delta_{i,\epsilon}^{-\frac{n}{q_0+1}}}{(1+\frac{|x-\xi_{i,\epsilon}|}{\delta_{i,\epsilon}})^{(n-2)p_0-2}}.
 \end{align*}
Hence we have
 \begin{align}\label{hi'}
h_i(x)\leq \frac{C\delta_{i,\epsilon}^{-\frac{n}{q_0+1}}}{(1+\frac{|x-\xi_{i,\epsilon}|}{\delta_{i,\epsilon}})^{(n-2)p_0-2}}.
 \end{align}

In view of \eqref{hi} and \eqref{hi'}, we have the following result analogous to Lemma \ref{lemP}.
 \begin{Lem}\label{lemP'2}
 There exists $c>0$ such that for all $x\in\Omega$,
  \begin{align}\label{0}
  \begin{split}
 &0\leq PU_i\leq U_i,\ \ \ 0\leq PV_i\leq V_i,\\
 &0\leq U_i- PU_i \leq C\delta_{i,\epsilon}^{\frac{p_0n}{q_0+1}}\eta_{i,\epsilon}^{-\frac{p_0+1}{q_0+1}n},\ \
0\leq U_i- PU_i \leq \frac{C\delta_{i,\epsilon}^{-\frac{n}{q_0+1}}}{(1+\frac{|x-\xi_{i,\epsilon}|}{\delta_{i,\epsilon}})^{(n-2)p_0-2}},
 \\
  &0\leq V_i- PV_i\leq\frac{b_{n,p_0}}{\gamma_n}\delta_{i,\epsilon}^{\frac{n}{q_0+1}}H(x,\xi_{i,\epsilon})
  \leq\frac{c_2\delta_{i,\epsilon}^{\frac{n}{q_0+1}}}{|x-\bar\xi_{i,\epsilon}|^{n-2}}.
 \end{split}\end{align}

 Moreover, there hold that
 \begin{align*}
 %&R_{1,\xi,\delta}(x):=PU_{\xi,\delta}-U_{\xi,\delta}+a_{n,p_0}\delta^{\frac{p_0n}{q_0+1}}\hat H(x,\xi),\\
 &R_{2,\xi_{i,\epsilon},\delta_{i,\epsilon}}(x):=PV_i-V_i
 +\frac{b_{n,p_0}}{\gamma_n}\delta_{i,\epsilon}^{\frac{n}{q_0+1}}H(x,\xi_{i,\epsilon}),\ \
 \|R_{2,\xi_{i,\epsilon},\delta_{i,\epsilon}}\|_{L^\infty(\Omega)}=O(\frac{\delta_{i,\epsilon}^{\frac{n}{q_0+1}+1}}{\eta_{i,\epsilon}^{n-1}}).\end{align*}
 \end{Lem}

\iffalse
For the same definition of \eqref{psiphi}, we also have
\begin{Lem}\label{lemP22}
For $i=1,\ldots,\kappa$ and $l=0,1,\ldots,n$, for $x\in\Omega$,
  \begin{align*}
 & P\Psi_{i}^l=\begin{cases}
 \Psi_{i}^l+O(\delta_{i,\epsilon}^{-1}U_i(x)),\ &l=0\\
  \Psi_{i}^l+O(\delta_{i,\epsilon}^{-1}\partial_{l}U_i(x)),\ &l=1,\dots,n,
 \end{cases}
 \end{align*}
and
  \begin{align*}
 & P\Phi_{i}^l=\begin{cases}
 \Phi_{i}^l+\frac{b_{n,q_0}}{\gamma_n}\delta_{i,\epsilon}^{\frac{n}{q_0+1}-1}H(x,\xi_{i,\epsilon})+o(\delta_{i,\epsilon}^{\frac{n}{q_0+1}-1}\eta_{i,\epsilon}^{-(n-2)}),\ &l=0\\
  \Psi_{i}^l+\frac{b_{n,q_0}}{\gamma_n}\delta_{i,\epsilon}^{\frac{n}{q_0+1}}\partial_{\xi,l}H(x,\xi_{i,\epsilon})+o(\delta_{i,\epsilon}^{\frac{n}{q_0+1}}\eta_{i,\epsilon}^{-(n-1)}),\ &l=1,\dots,n,
 \end{cases}
 \end{align*}
where $\partial_{\xi,l}$ is the $l-$th components of  $\nabla_{\xi_{i,\epsilon}}$.

 \end{Lem}
 \fi

 \medskip

In this case, due to the completely different exponent range from that in Lemma \ref{lemP'}, we must re-prove the integrals involving the gradient functions.

 \begin{Lem}\label{lemP'2'}
There exists $\sigma>0$ such that
\begin{align}\label{2'}
\|\nabla PU_i\|_{L^{\frac{p_0+1}{p_0}}(\Omega)}=O\left(\epsilon^{1+\sigma}\right),\ \
\|\nabla PV_i\|_{L^{\frac{q_0+1}{q_0}}(\Omega)}=O\left(\epsilon^{1+\sigma}\right).
\end{align}
Moreover,
\begin{align}\label{1'''}
\int_\Omega|\nabla PU_i|PV_j=O\left(\epsilon^{1+\sigma}\right),\ \ \int_\Omega|\nabla PV_i|PU_j=O\left(\epsilon^{1+\sigma}\right).
\end{align}

 \end{Lem}

 \begin{proof}

From the integral equation of $PU_i$ we know that
\begin{align*}
\nabla PU_i(x)=\int_\Omega\nabla_x\Big(\frac{\gamma_n}{|x-y|^{n-2}}-H(x,y)
\Big)V_i^{p_0}dy.
\end{align*}

First, we estimate \eqref{2'}.
From Hardy-Littlewood-Sobolev inequalities,  %for some $B_R(0)\supset\Omega$, we have
\begin{align*}
\|\nabla PU_i\|_{L^{\frac{p_0+1}{p_0}}(\Omega)}
&\leq C\Big\|\int \frac1{|x-y|^{n-1}}V_i^{p_0}(y)dy\Big\|_{L^{\frac{p_0+1}{p_0}}(\Omega)}\\
&\leq C\|V_i^{p_0}\|_{L^{r}}
=O\left( \delta_{i,\epsilon}^{\frac n{r}-\frac{p_0}{p_0+1}n}\right)
=O\left(\delta_{i,\epsilon}^{\frac{(n-2)p_0-2}{(n-2)p_0-1}+\sigma}\right),
\end{align*}
where we use  $1+\frac{p_0}{p_0+1}=\frac{n-1}{n}+\frac1{r}$  and by the assumption $p_0>p_n$ in \eqref{P}, we check easily that
$$\frac{(n-2)p_0-2}{n((n-2)p_0-1)}+\frac{p_0}{p_0+1}<\frac1r=\frac{p_0}{p_0+1}+\frac1n
<\frac{(n-2)p_0}n.$$

Next, we prove that
\begin{align*}
\|\nabla PV_i\|_{L^{\frac{q_0+1}{q_0}}(\Omega)}
&\leq C\Big\|\int_\Omega \frac1{|x-y|^{n-1}}U_i^{q_0}(y)dy\Big\|_{L^{\frac{q_0+1}{q_0}}(\Omega)}\\
&\leq C\|U_i^{q_0}\|_{L^{r}}
=O\left( \delta_{i,\epsilon}^{\frac n{r}-\frac{q_0}{q_0+1}n}\right)
=O\left(\delta_{i,\epsilon}^{\frac{(n-2)p_0-2}{n((n-2)p_0-1)}+\sigma}\right),
\end{align*}
where we estimate by noting that
$$\frac{(n-2)p_0-2}{(n-2)p_0-1}+\frac{q_0}{q_0+1}<\frac1r=\frac{q_0}{q_0+1}+\frac1n<\frac{(p_0+1)q_0}{q_0+1}.$$ 
\iffalse
such that
for any small $\theta_0>0$,
\begin{align*}
\|\nabla PV_i\|_{L^{\frac{q_0+1}{q_0}}(\Omega)}\leq C\left\|\frac1{|x|^{n-1}}\right\|_{L^{r_1}(B_M(0)}\|U_i^{q_0}\|_{L^{r}}
=O\left( \delta_{i,\epsilon}^{\frac n{r}-\frac{q_0}{q_0+1}n-\theta_0}\right)
=O\left(\delta_{i,\epsilon}^{\frac{n-2}{n-1}+\sigma}\right);
\end{align*}

While if  $((n-2)p_0-2)q_0r=\frac{(p_0+1)q_0nr}{q_0+1}<n$, we take $r_1<\frac n{n-1}$ and $\frac{q_0p_0}{q_0+1}<\frac1r<\frac{q_0}{q_0+1}+\frac1n$ such that under the condition
\eqref{p0small},
\begin{align*}
\|\nabla PV_i\|_{L^{\frac{q_0+1}{q_0}}(\Omega)}\leq C\left\|\frac1{|x|^{n-1}}\right\|_{L^{r_1}(B_M(0)}\|U_i^{q_0}\|_{L^{r}(\Omega)}
=O\left( \delta_{i,\epsilon}^{\frac{q_0p_0n}{q_0+1}}\right)
=O\left(\delta_{i,\epsilon}^{\frac{n-2}{n-1}+\sigma}\right).
\end{align*}
\fi
Then we conclude \eqref{2'}.

Moreover, we estimate that
\begin{align*}
\begin{split}
\int_\Omega|\nabla PU_i|PV_j\leq C\|\nabla PU_i\|_{L^{\frac{p_0+1}{p_0}}(\Omega)}\|PV_j\|_{L^{p_0+1}}=O(\epsilon^{1+\sigma}).
\end{split}\end{align*}

\iffalse
Also by setting $$\delta_{i,\epsilon}:=\delta\Lambda_i,$$ applying Hardy-Littlewood-Sobolev inequalities, $\frac1{r'}+\frac1r+1=2$, we have
\begin{align}\label{1'}
\begin{split}
\int_\Omega|\nabla PU_i|PV_j&\leq\int_\Omega\int_\Omega V_j(x)\frac1{|x-y|^{n-1}}V_i^{p_0}(y)dy\\
&\leq C\|V_j\|_{L^{r'}(\Omega)}\left\|\frac1{|x|^{n-1}}\right\|_{L^r(B_M(0)}\|V_i^{p_0}\|_{L^1(\Omega)}\\
&=O\left(\delta_{i,\epsilon}^{\frac n{r'}-\frac n{p_0+1}}\delta_{i,\epsilon}^{n-\frac{ np_0}{p_0+1}}\right)=O\left(\delta_{i,\epsilon}^{\frac n{r'}}\right).
\end{split}\end{align}
where we choose $r<\frac n{n-1}$ to ensure $\|\frac1{|x|^{n-1}}\|_{L^r(B_M(0)}<\infty$.
If we further assume $r>\frac{n(n-1)}{(n-1)^2+1}$, then $\frac n{r'}=n(1-\frac1r)>\frac{n-2}{n-1}$, that is there exists $\sigma>0$ such that
 $\frac n{r'}=\frac{n-2}{n-1}+\sigma$.
\fi
 Similarly,
 \begin{align*} 
\begin{split}
&\int_\Omega|\nabla PV_i|PU_j=O(\epsilon^{1+\sigma}),
\end{split}\end{align*}
which is \eqref{1'''}.

 \end{proof}
\medskip

\medskip

\medskip

\section{Function Spaces and Problem Setting}
 \iffalse
Having \eqref{emb} in mind, we introduce the following Banach space
\begin{align*}
&X_{p_0,q_0}:=\Big\{\Big(W^{2,\frac{p_0+1}{p_0}}(\Omega)\cap W_0^{1,p^*}(\Omega)\Big)\times \Big(W^{2,\frac{q_0+1}{q_0}}(\Omega)\cap W_0^{1,q^*}(\Omega)\Big)
\Big\}\\
& X_{p,q}:=\{(v_1,v_2)\in X_{p_0,q_0}: v_1\in L^{q_0+1-\beta\epsilon}(\Omega), v_2\in L^{p_0+1-\alpha\epsilon}(\Omega)
\}
\end{align*}
which is equipped with the norm
\begin{align*}
\|(v_1,v_2)\|_{X_{p,q,\epsilon}}=\|\Delta v_1\|_{L^{\frac{p_0+1}{p_0}}(\Omega)}+\|\Delta v_2\|_{L^{\frac{q_0+1}{q_0}}(\Omega)}
+\|v_1\|_{L^{q_0+1-\beta\epsilon}(\Omega)}+\|v_2\|_{L^{p_0+1-\alpha\epsilon}(\Omega)}
\end{align*}
\fi

Having \eqref{emb} in mind, we introduce the following Banach space
\begin{align*}
&X_{p_0,q_0}:=W_0^{1,p^*}(\Omega)\times W_0^{1,q^*}(\Omega)\hookrightarrow  L^{q_0+1}(\Omega)\times L^{p_0+1}(\Omega), \\
& X_{p,q}:=\{(v_1,v_2)\in X_{p_0,q_0}: v_1\in L^{q_0+1-\beta\epsilon}(\Omega), v_2\in L^{p_0+1-\alpha\epsilon}(\Omega)
\}
\end{align*}
which is equipped with the norm
\begin{align*}
\|(v_1,v_2)\|_{X_{p,q,\epsilon}}=\|\nabla v_1\|_{L^{p^*}(\Omega)}+\|\nabla v_2\|_{L^{q^*}(\Omega)}
+\|v_1\|_{L^{q+1}(\Omega)}+\|v_2\|_{L^{p+1}(\Omega)}.
\end{align*}

Since $a(x)$ is strictly positive and bounded in $\overline\Omega$,
the following functional is well-defined in $ X_{p,q}$:
\begin{align}\label{I}
I_\epsilon(u,v):=\int_{\Omega}a(x)\nabla u\cdot\nabla vdx
-\frac1{p+1}\int_{\Omega}a(x)| v|^{p+1}dx-\frac1{q+1}\int_{\Omega}a(x)|u|^{q+1}dx.
\end{align}
Moreover, we can take the equivalent norm of $X_{p,q}$ as
\begin{align*}
\|(v_1,v_2)\|:=\Big(\int_{\Omega}a(x)|\nabla v_1|^{p^*}\Big)^{\frac1{p^*}}+\Big(\int_{\Omega}a(x)|\nabla v_2|^{q^*}\Big)^{\frac1{q^*}}
\end{align*}
and the quadratic part
\begin{align*}
(u,v)_a:=\int_{\Omega}a(x)\nabla u\cdot\nabla vdx
\end{align*} of $I$ satisfies
\begin{align*}
|(u,v)_a|\leq C\|(u,v)\|^2.
\end{align*}

Denote by $i^*$ the formal adjoin operator of the embedding $i:X_{p_0,q_0}\hookrightarrow L^{q_0+1}(\Omega)\times L^{p_0+1}(\Omega)$. Then, by the Calder\'on-Zygmund estimate, the operator $i^*$ maps $L^{\frac{q_0+1}{q_0}}(\Omega)\times L^{\frac{p_0+1}{p_0}}(\Omega)$ to $X_{p_0,q_0}$, and we say $i^*(w_1,w_2)=(v_1,v_2)$
if and only if
\begin{align*}
\begin{cases}
-div(a(x)\nabla v_1)=a(x)w_2\ \ &in\ \Omega,\vspace{2mm}\\
-div(a(x)\nabla v_2)=a(x) w_1\ \ &in\ \Omega,\vspace{2mm}\\
v_1=v_2=0\ \ &on\ \partial\Omega,
\end{cases}
\end{align*}
or equivalently,  for all $\phi,\psi\in C_c^\infty(\Omega)$,
\begin{align*}
\begin{cases}
( v_1,\phi)_a=\dis\int_\Omega a(x)w_2\phi dx,\vspace{2mm}\\
(v_2,\psi)_a=\dis\int_\Omega a(x) w_1\psi dx.
\end{cases}
\end{align*}

Rewrite problem \eqref{eqv} as
\begin{align}\label{i*}
(v_1,v_2)=i^*(|v_1|^{q-1}v_1,|v_2|^{p-1}v_2).
\end{align}

The solutions of \eqref{eqv} in Theorem \ref{th1} are in fact of the form
\begin{align}\label{constructv'}
v_{1,\epsilon}=\sum_{i=1}^\kappa PU_i+\psi,\ \ \
v_{2,\epsilon}=\sum_{i=1}^\kappa PV_i+\phi,
\end{align}
where we recall $U_i=U_{\xi_{i,\epsilon},\delta_{i,\epsilon}},V_i=V_{\xi_{i,\epsilon},\delta_{i,\epsilon}}$ with
 $\delta_{i,\epsilon}$ and $\xi_{i,\epsilon}$ denoted by \eqref{set}.
 \vskip 0.5cm

Let $$W_1=W_{1,\xi,\Lambda,t}=\sum_{i=1}^\kappa PU_i,\ \ \ \ W_2=W_{2,\xi,\Lambda,t}=\sum_{i=1}^\kappa PV_i.$$
Recall $$\Psi^j_i=\Psi_{\xi_{i,\epsilon},\delta_{i,\epsilon}}^j,\ \ \ \Phi^j_i=\Phi_{\xi_{i,\epsilon},\delta_{i,\epsilon}}^j.$$    %%%%%%%%%验证正交的定义，不是内积空间的情况下
Set the spaces
\begin{align*}
K_{\xi,\Lambda,t}=span\Big\{(P\Psi_i^j, P\Phi_i^j),i=1,\ldots,\kappa, j=0,1,\ldots,n\Big\},
\end{align*}
\begin{align*}
E_{\xi,\Lambda,t}=\Big\{(\psi,\phi)\in X_{p,q}:\int_\Omega a(x)\left(\nabla P\Phi_i^j\cdot\nabla\psi+\nabla P\Psi_i^j\cdot\nabla\phi\right)
=0,i=1,\ldots,\kappa, j=0,1,\ldots,n\Big\}.
\end{align*}
Moreover, we introduce the orthogonal projection operators $\Pi_{\xi,\Lambda,t}$ and $\Pi^\bot_{\xi,\Lambda,t}$ in $X_{p,q}$ with ranges $K_{\xi,\Lambda,t}$
and $E_{\xi,\Lambda,t}$ respectively.
\medskip

To solve \eqref{i*}, we are to find $(\vec\xi,\vec\Lambda,\vec t)\in\Gamma$ and some function pair $(\psi,\phi)\in E_{\xi,\Lambda,t}$ such that
\begin{align}\label{bot}
\Pi^\bot_{\xi,\Lambda,t}\left((W_1+\psi,W_2+\phi)-i^*(|W_1+\psi|^{q-1}(W_1+\psi),|W_2+\phi|^{p-1}(W_2+\phi))\right)=0
\end{align}
and
\begin{align}\label{Pi}
\Pi_{\xi,\Lambda,t}\left((W_1+\psi,W_2+\phi)-i^*(|W_1+\psi|^{q-1}(W_1+\psi),|W_2+\phi|^{p-1}(W_2+\phi))\right)=0.
\end{align}
We carry out the reduction method in two steps. First, for given $(\xi,\Lambda,t)\in\Gamma$ and $\epsilon$ small, we find a pair
$(\psi,\phi)\in E_{\xi,\Lambda,t}$ such that \eqref{bot} holds. Second, a reduced problem is proved.

\medskip

\section{Finite-dimensional reduction}
\subsection{Linear Problem}

The linear operator $L_{\xi,\Lambda,t}: E_{\xi,\Lambda,t}\rightarrow E_{\xi,\Lambda,t}$ is defined as
\begin{align*}
L_{\xi,\Lambda,t}(\psi,\phi):=(\psi,\phi)-\Pi^\bot_{\xi,\Lambda,t}i^*(q_0W_1^{q_0-1}\psi,p_0W_2^{p_0-1}\phi)
\end{align*}

We first prove the following statement.
\begin{Prop}\label{propL}
For any compact subset $\Gamma_0$ of $\Gamma$, there exist $\epsilon_0>0$ and $C>0$ such that for any $\epsilon\in(0,\epsilon_0)$ and $(\xi,\Lambda,t)\in\Gamma_0$, the operator
$L_{\xi,\Lambda,t}$ is invertible and it holds that
\begin{align}\label{invert}
\|L_{\xi,\Lambda,t}(\psi,\phi)\|\geq C\|(\psi,\phi)\|,\ \ \forall(\psi,\phi)\in E_{\xi,\Lambda,t}.
\end{align}

\end{Prop}

\begin{proof}
Suppose that \eqref{invert} is not true. There exist sequences $\{\epsilon_m\}_m$ of small positive numbers,
$$\{(\vec\xi_m,\vec \Lambda_m,\vec t_m)=(\xi_{1,m},\ldots,\xi_{\kappa,m},\Lambda_{1,m},\ldots,\Lambda_{\kappa,m},t_{1,m},\ldots,t_{\kappa,m}\}_m
\subset\Gamma,\ \ (\psi_m,\phi_m)\in E_{\xi_m,\Lambda_m,t_m}$$
and
$$(h_{1,m},h_{2,m}):=L_{\xi_m,\Lambda_m,t_m}(\psi_m,\phi_m)$$
such that $\epsilon_m\rightarrow0, (\vec\xi_m,\vec \Lambda_m,\vec t_m)\rightarrow(\vec\xi_\infty,\vec \Lambda_\infty,\vec t_\infty)\in\Gamma$ as $n\rightarrow\infty$,
$$\|(\psi_m,\phi_m)\|=1,\ \ \ \|(h_{1,m},h_{2,m})\|\rightarrow0,\ as\ n\rightarrow\infty.$$
Set $$\delta_{i,\epsilon}=\delta \Lambda_i,\ \ \ 
\eta=\min\Big\{dist(\xi_{i,\epsilon},\partial\Omega), \frac{|\xi_{i,\epsilon}-\xi_{j,\epsilon}|}2, i,j=1,\ldots,\kappa, i\neq j\Big\}.$$

Then from the assumptions, there exist some $c_{i,l,m}$ such that
\begin{align*}
(\psi_m,\phi_m)-i^*(q_0W_1^{q_0-1}\psi,p_0W_2^{p_0-1}\phi)=(h_{1,m},h_{2,m})+\sum_{i=1}^\kappa\sum_{l=0}^nc_{i,l,m}(P\Psi_{i,m}^l,P\Phi_{i,m}^l).
\end{align*}
In other words,
\begin{align}\label{eql}
\begin{cases}
&\displaystyle-div(a(x)\nabla\psi_m)=a(x)q_0W_2^{p_0-1}\phi_m-div(a(x)\nabla h_{1,m})-\sum_{i=1}^\kappa\sum_{l=0}^nc_{i,l,n}div(a(x)\nabla P\Psi_{i,m}^l)\\
&\displaystyle-div(a(x)\nabla\phi_m)=a(x)q_0W_1^{q_0-1}\psi_m-div(a(x)\nabla h_{2,m})-\sum_{i=1}^\kappa\sum_{l=0}^nc_{i,l,n}div(a(x)\nabla P\Phi_{i,m}^l).\\
\end{cases}
\end{align}

First, we claim
\begin{align}\label{claim1}
\sum_{i=1}^\kappa\sum_{l=0}^n|c_{i,l,n}|=o(\delta)=o(\epsilon^{\frac{n-1}{n-2}}).
\end{align}

In fact, for $j=1,\ldots,\kappa$ and $s=0,1,\dots,n$, we test \eqref{eql} with $(P\Phi_{j,m}^s,P\Psi_{j,m}^s)\in L^{p_0+1}(\Omega)\times L^{q_0+1}(\Omega)$ to get
\begin{align*}
&\int_\Omega a(x)\nabla P\Phi_{j,m}^s\nabla\psi_m+a(x)\nabla P\Psi_{j,m}^s\nabla\phi_m\\
=&\int_\Omega a(x)p_0W_2^{p_0-1}\phi_m P\Phi_{j,m}^s+q_0W_1^{q_0-1}\psi_m P\Psi_{j,m}^s\\
&+\int_\Omega a(x)\nabla P\Phi_{j,m}^s\nabla h_{1,m}+a(x)\nabla P\Psi_{j,m}^s\nabla h_{2,m}\\
&+\sum_{i=1}^\kappa\sum_{l=0}^nc_{i,l,n}\int_\Omega a(x)(\nabla P\Phi_{i,m}^l\nabla P\Psi_{j,m}^s+\nabla P\Psi_{i,m}^l\nabla P\Phi_{j,m}^s).
\end{align*}
Using \eqref{eqP}, we obtain
\begin{align*}
\begin{split}
&\int_\Omega a(x)p_0\phi_mP\Phi_{j,m}^s \left(V_j^{p_0-1}-W_2^{p_0-1}\right)
+a(x)q_0\psi_mP\Psi_{j,m}^s\left(U_j^{q_0-1}-W_1^{q_0-1} \right)\\
=&\int_\Omega \nabla a(x)(\nabla P\Psi_{j,m}^s\phi_m+\nabla P\Phi_{j,m}^s\psi_m)
+\int_\Omega a(x)(\nabla P\Phi_{j,m}^s\nabla h_{1,m}+\nabla P\Psi_{j,m}^s\nabla h_{2,m})\\
&+\sum_{i=1}^\kappa\sum_{l=0}^nc_{i,l,n}\int_\Omega a(x)(\nabla P\Phi_{i,m}^l\nabla P\Psi_{j,m}^s+\nabla P\Psi_{i,m}^l\nabla P\Phi_{j,m}^s).
\end{split}
\end{align*}
We denote the left-hand side by $J_L$, and the first, second and the third integral by $J_1,J_2$ and $J_3$ respectively, and estimate each one as follows.
%For simplicity, we denote $\delta_{i,\epsilon}=\delta \Lambda_i=\epsilon^{\frac{n-1}{n-2}}\Lambda_i$.
\medskip

First, we deal with  $J_L$ in the case of $p_0>\frac n{n-2}$.

In fact,
\begin{align*}
\begin{split}
&\int_\Omega a(x)p_0\phi_mP\Phi_{j,m}^s \left(V_j^{p_0-1}-W_2^{p_0-1}\right)
+a(x)q_0\psi_mP\Psi_{j,m}^s\left(U_j^{q_0-1}-W_1^{q_0-1} \right)\\
&\leq C\|(\psi_m,\phi_m)\|\Big(\Big\|P\Phi_{j,m}^s \left(V_j^{p_0-1}-W_2^{p_0-1}\right)\Big\|_{L^{\frac{p_0+1}{p_0}}(\Omega)}
+\Big\|P\Psi_{j,m}^s \left(U_j^{q_0-1}-W_1^{q_0-1}\right)\Big\|_{L^{\frac{q_0+1}{q_0}}(\Omega)}\Big)\\
&\leq C\sum_{l=1}^\kappa\Big(\Big\|P\Phi_{j,m}^s \left(V_j^{p_0-1}-W_2^{p_0-1}\right)\Big\|_{L^{\frac{p_0+1}{p_0}}(B_\eta(\xi_{l,m}))}\\&\quad
+\Big\|P\Psi_{j,m}^s \left(U_j^{q_0-1}-W_1^{q_0-1}\right)\Big\|_{L^{\frac{q_0+1}{q_0}}(B_\eta(\xi_{l,m}))}\Big) +\delta^{-1}O\Big(\Big(\frac{\delta}{\eta}\Big)^{\frac{np_0}{q_0+1}}+\Big(\frac{\delta}{\eta}\Big)^{\frac{nq_0}{p_0+1}}\Big).
\end{split}\end{align*}
For $l\neq j$,
\begin{align}\label{JL1}
\begin{split}
&\Big\|P\Phi_{j,m}^s \left(V_j^{p_0-1}-W_2^{p_0-1}\right)\Big\|_{L^{\frac{p_0+1}{p_0}}(B_\eta(\xi_{l,m}))}
+\Big\|P\Psi_{j,m}^s \left(U_j^{q_0-1}-W_1^{q_0-1}\right)\Big\|_{L^{\frac{q_0+1}{q_0}}(B_\eta(\xi_{l,m}))}\\
\leq& C\|V_l^{p_0-1}\|_{L^{\frac{p_0+1}{p_0-1}}(B_\eta(\xi_{l,m}))}\|P\Phi_{j,m}^s\|_{L^{p_0+1}(B_\eta(\xi_{l,m}))}\\&\quad
+C\|U_l^{q_0-1}\|_{L^{\frac{q_0+1}{q_0-1}}(B_\eta(\xi_{l,m}))}\|P\Psi_{j,m}^s\|_{L^{q_0+1}(B_\eta(\xi_{l,m}))}+\delta^{-1}O\Big(\Big(\frac{\delta}{\eta}\Big)^{\frac{np_0}{q_0+1}}+\Big(\frac{\delta}{\eta}\Big)^{\frac{nq_0}{p_0+1}}\Big)\\
\leq& C%\Big(\int_{B_\eta(\xi_{l,m})}\frac{\delta^{-n}}{}dx\Big)^{\frac{p_0-1}{p_0+1}}
\delta^{-1}
O\Big(\Big(\frac{\delta}{\eta}\Big)^{(n-2)(p_0+1)-n}+\Big(\frac{\delta}{\eta}\Big)^{(n-2)(q_0+1)-n}\Big)
+\delta^{-1}O\Big(\Big(\frac{\delta}{\eta}\Big)^{\frac{np_0}{q_0+1}}+\Big(\frac{\delta}{\eta}\Big)^{\frac{nq_0}{p_0+1}}\Big)
\\=&o(\delta^{-1}).
\end{split}
\end{align}
While for $l=j$,
\begin{align}\label{JL2}
\begin{split}
&\Big\|P\Phi_{j,m}^s \left(V_j^{p_0-1}-W_2^{p_0-1}\right)\Big\|_{L^{\frac{p_0+1}{p_0}}(B_\eta(\xi_{j,m}))}
+\Big\|P\Psi_{j,m}^s \left(U_j^{q_0-1}-W_1^{q_0-1}\right)\Big\|_{L^{\frac{q_0+1}{q_0}}(B_\eta(\xi_{j,m}))}\\
\leq& C\sum_{i\neq j}\|V_i^{p_0-1}\|_{L^{\frac{p_0+1}{p_0-1}}(B_\eta(\xi_{j,m}))}\|P\Phi_{j,m}^s\|_{L^{p_0+1}(B_\eta(\xi_{j,m}))}\\&
+C\sum_{i\neq j}\|U_i^{q_0-1}\|_{L^{\frac{q_0+1}{q_0-1}}(B_\eta(\xi_{j,m}))}\|P\Psi_{j,m}^s\|_{L^{q_0+1}(B_\eta(\xi_{j,m}))}+\delta^{-1}\Big(\Big(\frac{\delta}{\eta}\Big)^{\frac{np_0}{q_0+1}}+\Big(\frac{\delta}{\eta}\Big)^{\frac{nq_0}{p_0+1}}\Big)\\
\leq& C%\Big(\int_{B_\eta(\xi_{l,m})}\frac{\delta^{-n}}{}dx\Big)^{\frac{p_0-1}{p_0+1}}
\delta^{-1}O\Big(\Big(\frac{\delta}{\eta}\Big)^{(n-2)(p_0+1)-n}+\Big(\frac{\delta}{\eta}\Big)^{(n-2)(q_0+1)-n}\Big)+\delta^{-1}O\Big(\Big(\frac{\delta}{\eta}\Big)^{\frac{np_0}{q_0+1}}+\Big(\frac{\delta}{\eta}\Big)^{\frac{nq_0}{p_0+1}}\Big)
\\=&o(\delta^{-1}).
\end{split}
\end{align}

While in the case of $p_0<\frac n{n-2}$, we have
\begin{align*}
\begin{split}
&\int_\Omega a(x)p_0\phi_mP\Phi_{j,m}^s \left(V_j^{p_0-1}-W_2^{p_0-1}\right)
+a(x)q_0\psi_mP\Psi_{j,m}^s\left(U_j^{q_0-1}-W_1^{q_0-1} \right)\\
%&\leq C\|(\psi_m,\phi_m)\|\Big(\Big\|P\Phi_{j,m}^s \left(V_j^{p_0-1}-W_2^{p_0-1}\right)\Big\|_{L^{\frac{p_0+1}{p_0}}(\Omega)}
%+\Big\|P\Psi_{j,m}^s \left(U_j^{q_0-1}-W_1^{q_0-1}\right)\Big\|_{L^{\frac{q_0+1}{q_0}}(\Omega)}\Big)\\
&\leq C\sum_{l=1}^\kappa\Big(\Big\|P\Phi_{j,m}^s \left(V_j^{p_0-1}-W_2^{p_0-1}\right)\Big\|_{L^{\frac{p_0+1}{p_0}}(B_\eta(\xi_{l,m}))}\\&\quad
+\Big\|P\Psi_{j,m}^s \left(U_j^{q_0-1}-W_1^{q_0-1}\right)\Big\|_{L^{\frac{q_0+1}{q_0}}(B_\eta(\xi_{l,m}))}\Big)
+\delta^{-1}O\Big(\Big(\frac{\delta}{\eta}\Big)^{\frac{np_0}{q_0+1}}+\Big(\frac{\delta}{\eta}\Big)^{\frac{np_0q_0}{q_0+1}}\Big).
\end{split}\end{align*}
For $l\neq j$,
\begin{align}\label{JL1}
\begin{split}
&\Big\|P\Phi_{j,m}^s \left(V_j^{p_0-1}-W_2^{p_0-1}\right)\Big\|_{L^{\frac{p_0+1}{p_0}}(B_\eta(\xi_{l,m}))}
+\Big\|P\Psi_{j,m}^s \left(U_j^{q_0-1}-W_1^{q_0-1}\right)\Big\|_{L^{\frac{q_0+1}{q_0}}(B_\eta(\xi_{l,m}))}\\
\leq& C\|V_l^{p_0-1}\|_{L^{\frac{p_0+1}{p_0-1}}(B_\eta(\xi_{l,m}))}\|P\Phi_{j,m}^s\|_{L^{p_0+1}(B_\eta(\xi_{l,m}))}\\&\quad
+C\|U_l^{q_0-1}\|_{L^{\frac{q_0+1}{q_0-1}}(B_\eta(\xi_{l,m}))}\|P\Psi_{j,m}^s\|_{L^{q_0+1}(B_\eta(\xi_{l,m}))}
+\delta^{-1}O\Big(\Big(\frac{\delta}{\eta}\Big)^{\frac{np_0}{q_0+1}}+\Big(\frac{\delta}{\eta}\Big)^{\frac{np_0q_0}{q_0+1}}\Big)\\
\leq& C%\Big(\int_{B_\eta(\xi_{l,m})}\frac{\delta^{-n}}{}dx\Big)^{\frac{p_0-1}{p_0+1}}
\delta^{-1}O\Big(\Big(\frac{\delta}{\eta}\Big)^{\frac{n}{q_0+1}}+\Big(\frac{\delta}{\eta}\Big)^{\frac{np_0}{q_0+1}}\Big)
+\delta^{-1}O\Big(\Big(\frac{\delta}{\eta}\Big)^{\frac{np_0}{q_0+1}}+\Big(\frac{\delta}{\eta}\Big)^{\frac{np_0q_0}{q_0+1}}\Big)
\\=&o(\delta^{-1}).
\end{split}
\end{align}
While for $l=j$,
\begin{align}\label{JL2}
\begin{split}
&\Big\|P\Phi_{j,m}^s \left(V_j^{p_0-1}-W_2^{p_0-1}\right)\Big\|_{L^{\frac{p_0+1}{p_0}}(B_\eta(\xi_{j,m}))}
+\Big\|P\Psi_{j,m}^s \left(U_j^{q_0-1}-W_1^{q_0-1}\right)\Big\|_{L^{\frac{q_0+1}{q_0}}(B_\eta(\xi_{j,m}))}\\
\leq& C\sum_{i\neq j}\|V_i^{p_0-1}\|_{L^{\frac{p_0+1}{p_0-1}}(B_\eta(\xi_{j,m}))}\|P\Phi_{j,m}^s\|_{L^{p_0+1}(B_\eta(\xi_{j,m}))}\\&
+C\sum_{i\neq j}\|U_i^{q_0-1}\|_{L^{\frac{q_0+1}{q_0-1}}(B_\eta(\xi_{j,m}))}\|P\Psi_{j,m}^s\|_{L^{q_0+1}(B_\eta(\xi_{j,m}))}\\&
+\delta^{-1}O\Big(\Big(\frac{\delta}{\eta}\Big)^{\frac{n(p_0-1)}{q_0+1}}
+\Big(\frac{\delta}{\eta}\Big)^{n-2}+\Big(\frac{\delta}{\eta}\Big)^{\frac{np_0(q_0-1)}{q_0+1}}+\Big(\frac{\delta}{\eta}\Big)^{\frac{n(p_0+1)}{q_0+1}}\Big)
+\delta^{-1}O\Big(\Big(\frac{\delta}{\eta}\Big)^{\frac{np_0}{q_0+1}}+\Big(\frac{\delta}{\eta}\Big)^{\frac{np_0q_0}{q_0+1}}\Big)\\
=&\delta^{-1}O\Big(\Big(\frac{\delta}{\eta}\Big)^{\frac{n(p_0-1)}{q_0+1}}
+\Big(\frac{\delta}{\eta}\Big)^{n-2}+\Big(\frac{\delta}{\eta}\Big)^{\frac{np_0(q_0-1)}{q_0+1}}+\Big(\frac{\delta}{\eta}\Big)^{\frac{n(p_0+1)}{q_0+1}}\Big)
\\=&o(\delta^{-1}),
\end{split}
\end{align}
where we have used the estimate \eqref{hi}.% that $h_j(x)\leq \frac {\delta^{\frac{ Cp_0n}{q_0+1}}}{\eta^{(n-2)p_0-2}}$.

\medskip
Next, for $J_1$ and $J_2$, considering the proof of  Lemma \ref{lemP'} (for $p_0>\frac n{n-2}$) and Lemma \ref{lemP'2} (for $p_0<\frac n{n-2}$), we can obtain that
\begin{align}\label{J1}
\begin{split}
&\int_\Omega \nabla a(x)(\nabla P\Psi_{j,m}^s\phi_m+\nabla P\Phi_{j,m}^s\psi_m)\\
\leq &C\left(\|\nabla P\Psi_{j,m}^s\|_{L^{\frac{p_0+1}{p_0}}}\|\phi_m\|_{L^{p_0+1}}+\|\nabla P\Phi_{j,m}^s\|_{L^{\frac{q_0+1}{q_0}}}\|\psi_m\|_{L^{q_0+1}}\right)\\
=&o(\delta^{-1}).
\end{split}
\end{align}
By the assumption, $\|(h_{1,m},h_{2,m})\|=o(1)$ and then
\begin{align}\label{J2}
\begin{split}
&\int_\Omega a(x)(\nabla P\Phi_{j,m}^s\nabla h_{1,m}+\nabla P\Psi_{j,m}^s\nabla h_{2,m})\\
\leq& C\left(\|\nabla P\Psi_{j,m}^s\|_{L^{p^*}}\|\nabla h_{2,m}\|_{L^{q^*}}+\|\nabla P\Phi_{j,m}^s\|_{L^{q^*}}\|h_{1,m}\|_{L^{q^*}}\right)\\
=&o(\delta^{-1}).
\end{split}
\end{align}

On  the other hand, for $J_3$, we have
\begin{align}\label{J3}
\begin{split}
&\sum_{i=1}^\kappa\sum_{l=0}^nc_{i,l,n}\int_\Omega a(x)(\nabla P\Phi_{i,m}^l\nabla P\Psi_{j,m}^s+\nabla P\Psi_{i,m}^l\nabla P\Phi_{j,m}^s)\\
 =&-\sum_{i=1}^\kappa\sum_{l=0}^nc_{i,l,n}\int_\Omega\nabla a(x)(\nabla P\Phi_{i,m}^lP\Psi_{j,m}^s+\nabla P\Psi_{i,m}^l P\Phi_{j,m}^s)\\
&+\sum_{i=1}^\kappa\sum_{l=0}^nc_{i,l,n}\Big(\int_\Omega a(x)(-\Delta P\Phi_{i,m}^l P\Psi_{j,m}^s-\Delta P\Psi_{i,m}^lP\Phi_{j,m}^s)+o(1)\Big)\\
 =&-\sum_{i=1}^\kappa\sum_{l=0}^nc_{i,l,n}\int_\Omega\nabla a(x)(\nabla P\Phi_{i,m}^lP\Psi_{j,m}^s+\nabla P\Psi_{i,m}^l P\Phi_{j,m}^s)\\
&+\delta^{-2}\sum_{i=1}^\kappa\sum_{l=0}^nc_{i,l,n}\Big(\mathbb\delta_{ij}\mathbb\delta_{ls}\Lambda_j^{-2}\int_\Omega a(x)
(q_0U^{q_0-1}\Psi^s+p_0V^{p_0-1}P\Phi^s)+o(1)\Big)\\
 =&o(\delta^{-1})+\delta^{-2}\sum_{i=1}^\kappa\sum_{l=0}^nc_{i,l,n}\Big(\mathbb\delta_{ij}\mathbb\delta_{ls}\Lambda_j^{-2}\int_\Omega a(x)
(q_0U^{q_0-1}\Psi^s+p_0V^{p_0-1}P\Phi^s)+o(1)\Big).
\end{split}
\end{align}
Hence combining \eqref{JL1}, \eqref{JL2}, \eqref{J1}, \eqref{J2} and \eqref{J3}, we have shown \eqref{claim1}.

\vskip 0.5cm

\begin{Lem}\label{lem0}
There holds that
\begin{align*}
 \|W_2^{p_0-1}\phi_m \|_{L^{\frac{p_0+1}{p_0}}}+ \|W_1^{q_0-1}\psi_m \|_{L^{\frac{q_0+1}{q_0}}}\rightarrow0\ \ as\ \ m\rightarrow\infty.
\end{align*}

\end{Lem}

\begin{proof}
We only sketch the steps of the proof, which is similar to \cite{KP}.

Step 1. We take a smooth cut-off function
\begin{align*}
\chi(x)=\begin{cases}1,\ &in\ B_\rho(\xi_{l}),\\
0,\  &in\ \Omega\setminus  B_{2\rho}(\xi_{l}),
\end{cases}
\ \ \ \ \ \ |\nabla\chi(x)|\leq\frac2\rho,\ \ \ \ \  |\nabla^2\chi(x)|\leq\frac4{\rho^2}.
\end{align*}
Then for $y\in\Omega_{l,m}:=\frac{\Omega-\xi_{l,m}}{\delta_{l,m}}$,  we set
\begin{align*}
(\tilde\psi_m(y),\tilde\phi_m(y))=\left(\delta_{l,m}^{\frac n{q_0+1}}(\chi\psi_m)(\delta_{l,m}y+\xi_{l,m}),\delta_{l,m}^{\frac n{p_0+1}}(\chi\phi_m)(\delta_{l,m}y+\xi_{l,m})\right),
\end{align*}
which satisfies
\begin{align*}
(\tilde\psi_m(y),\tilde\phi_m(y))\rightarrow(\tilde\psi(y),\tilde\phi(y))
\end{align*}
weakly in $W_0^{1,p^*}(\R^n)\times W_0^{1,q^*}(\R^n)$, strongly in $L^{q+1}(\R^n)\times L^{p_0+1}(\R^n)$ and almost everywhere in $\R^n$.

Step 2. Taking $m\rightarrow\infty$ in \eqref{eql} and considering \eqref{claim1}, there holds that $(\tilde\psi,\tilde\phi)$ solves \eqref{eqlinear}.
Then applying Lemma \ref{lemnonde}, we obtain  $(\tilde\psi,\tilde\phi)=0$.

Step 3. We prove \begin{align*}
 \|W_2^{p_0-1}\phi_m \|_{L^{\frac{p_0+1}{p_0}}}\rightarrow0\ \ as\ \ m\rightarrow\infty,
\end{align*}
and
\begin{align*}
 \|W_1^{q_0-1}\psi_m \|_{L^{\frac{q_0+1}{q_0}}}\rightarrow0\ \ as\ \ m\rightarrow\infty.
\end{align*}
\end{proof}

Finally, we complete the proof of Proposition \ref{propL}. In fact, by the assumptions, Claim \eqref{claim1} and Lemma \ref{lem0}, then up to a subsequence,
\begin{align*}
1=\|(\psi_m,\phi_m)\|\leq& C \Big(\|W_2^{p_0-1}\phi_m \|_{L^{\frac{p_0+1}{p_0}}}+ \|W_1^{q_0-1}\psi_m \|_{L^{\frac{q_0+1}{q_0}}}\\&
\qquad+\|(h_{1,m},h_{2,m})\|+\delta^{-1}\sum_{i=1}^\kappa\sum_{l=0}^n|c_{i,l,n}|\Big)\rightarrow0,\ \ as\ m\rightarrow\infty,
\end{align*}
which is a contradiction, concluding the proof.

\end{proof}

                   %%%%%%%%%%%%%%%%%%%%%%%%%%%%%%%%%%%%%%%%%%%%%%%%%%%%%%%%%%%%%%%%%%%%%%%%%%%%%%%%%%%%%%%%%%%%%%%%%%%%%%%%%%%%%%%%%%%%%%前面等价范数可以再细一点解释

In view of Proposition \ref{propL}, the standard Fredholm alternative gives that
\begin{Cor} \label{corlinear}
For any $\epsilon\in(0,\epsilon_0)$, $(\vec\xi,\vec\Lambda,\vec t)\in\Gamma$ and $(h_1,h_2)\in E_{\xi,\Lambda,t}$, there exists solution
$(\psi,\phi)\in E_{\xi,\Lambda,t}$ to the linear problem $$L_{\xi,\Lambda,t}(\psi,\phi)=(h_1,h_2).$$
Moreover, $\|(h_1,h_2)\|\geq C\|(\psi,\phi)\|$.

\end{Cor}

\medskip

\subsection{Nonlinear Problem}

Applying Corollary \ref{corlinear}, we consider problem \eqref{bot}.
First, we rewrite \eqref{bot} as
\begin{align}\label{fixed}
\begin{split}
(\psi,\phi)=T_{\xi,\Lambda,t}(\psi,\phi):=L_{\xi,\Lambda,t}^{-1}\left(-R_\epsilon+N_\epsilon(\psi,\phi)\right),
\end{split}
\end{align}
where
\begin{align}\label{R}
\begin{split}
R_\epsilon=\Pi_{\xi,\Lambda,t}^\bot((W_1,W_2)-i^*(W_1^q,W_2^p))
\end{split}
\end{align}
and
\begin{align}\label{N}
\begin{split}
N_\epsilon(\psi,\phi)=\Pi_{\xi,\Lambda,t}^\bot i^*(N_{1,\epsilon}(\phi),N_{2,\epsilon}(\psi)),
\end{split}
\end{align}
with
\begin{align*}
\begin{split}
&N_{1,\epsilon}(\psi)=|W_1+\psi|^{q-1}(W_1+\psi)-W_1^q-q_0W_1^{q_0-1}\psi,\\
&N_{2,\epsilon}(\phi)=|W_2+\phi|^{p-1}(W_2+\phi)-W_2^p-p_0W_2^{p_0-1}\phi.
\end{split}
\end{align*}

\medskip

Next, we estimate the errors.
\begin{Lem}\label{lemerror}
For $\epsilon\in(0,\epsilon_0)$, there exists some $\sigma>0$ such that
\begin{align}
\|R_\epsilon\|=O\Big(\epsilon^{\frac12+\sigma}
\Big).
\end{align}

\end{Lem}

\begin{proof}
From \eqref{R}, the equivalent weighted space and  the definition of $i^*$, we get that
$$(U_i,V_i)=i^*\Big(U_i^{q_0}-\frac{\nabla a}a\cdot\nabla V_i,V_i^{p_0}-\frac{\nabla a}a\cdot\nabla U_i\Big),$$
which means
$$(W_1,W_2)=i^*\Big(\sum_{i=1}^\kappa U_i^{q_0}-\frac{\nabla a}a\cdot\nabla W_2,\sum_{i=1}^\kappa V_i^{p_0}-\frac{\nabla a}a\cdot\nabla W_1\Big).$$
Hence, by Lemma \ref{lemP'} and Lemma \ref{lemP'2} for different $p_0$-range respectively,  there exists some small $\sigma>0$ such that
\begin{align}\label{4.12}
\begin{split}
\|R_\epsilon\|\leq &C\Big(\Big\|\sum_{i=1}^\kappa U_i^{q_0}-W_1^{q_0}\Big\|_{L^{\frac{q_0+1}{q_0}}(\Omega)}
+\Big\|W_1^q-W_1^{q_0}\Big\|_{L^{\frac{q_0+1}{q_0}}(\Omega)}
+\Big\|\nabla a\cdot\nabla W_2\Big\|_{L^{\frac{q_0+1}{q_0}}(\Omega)}\\&
+\Big\|\sum_{i=1}^\kappa V_i^{p_0}-W_2^{p_0}\Big\|_{L^{\frac{p_0+1}{p_0}}(\Omega)}
+\Big\|W_2^p-W_2^{p_0}\Big\|_{L^{\frac{p_0+1}{p_0}}(\Omega)}
+\Big\|\nabla a\cdot\nabla W_1\Big\|_{L^{\frac{p_0+1}{p_0}}(\Omega)}
\Big)\\
\leq & C\Big(\Big\|\sum_{i=1}^\kappa U_i^{q_0}-W_1^{q_0}\Big\|_{L^{\frac{q_0+1}{q_0}}(\Omega)}
+\Big\|W_1^q-W_1^{q_0}\Big\|_{L^{\frac{q_0+1}{q_0}}(\Omega)}\\&
+\Big\|\sum_{i=1}^\kappa V_i^{p_0}-W_2^{p_0}\Big\|_{L^{\frac{p_0+1}{p_0}}(\Omega)}
+\Big\|W_2^p-W_2^{p_0}\Big\|_{L^{\frac{p_0+1}{p_0}}(\Omega)}
\Big)
+O(\epsilon^{1+\sigma})\\
:=&C(I+J)+O(\epsilon^{1+\sigma}),
\end{split}
\end{align}
where we denote
\begin{align*}
\begin{split}
&I:=\Big\|\sum_{i=1}^\kappa U_i^{q_0}-W_1^{q_0}\Big\|_{L^{\frac{q_0+1}{q_0}}(\Omega)}
+\Big\|\sum_{i=1}^\kappa V_i^{p_0}-W_2^{p_0}\Big\|_{L^{\frac{p_0+1}{p_0}}(\Omega)}\\
&J:=\Big\|W_2^p-W_2^{p_0}\Big\|_{L^{\frac{p_0+1}{p_0}}(\Omega)}+\Big\|W_1^q-W_1^{q_0}\Big\|_{L^{\frac{q_0+1}{q_0}}(\Omega)}.
\end{split}
\end{align*}

Recall the notation $$\eta=\min\Big\{dist(\xi_{i,\epsilon},\partial\Omega), \frac{|\xi_{i,\epsilon}-\xi_{j,\epsilon}|}2, i,j=1,\ldots,\kappa, i\neq j\Big\}.$$
We estimate
\begin{align}\label{4.13}
\begin{split}
I=&\Big\|\sum_{i=1}^\kappa U_i^{q_0}-W_1^{q_0}\Big\|_{L^{\frac{q_0+1}{q_0}}(\Omega)}+\Big\|\sum_{i=1}^\kappa V_i^{p_0}-W_2^{p_0}\Big\|_{L^{\frac{p_0+1}{p_0}}(\Omega)}\\
\leq& C\Big(\sum_{j=1}^\kappa\Big\|\sum_{i=1}^\kappa U_i^{q_0}-W_1^{q_0}\Big\|_{L^{\frac{q_0+1}{q_0}}(B_\eta(\xi_{j,\epsilon}))}
+\Big\|\sum_{i=1}^\kappa U_i^{q_0}-W_1^{q_0}\Big\|_{L^{\frac{q_0+1}{q_0}}(\Omega\setminus \cup_{j=1}^\kappa B_\eta(\xi_{j,\epsilon}))}\\
&+\sum_{j=1}^\kappa\Big\|\sum_{i=1}^\kappa V_i^{p_0}-W_2^{p_0}\Big\|_{L^{\frac{p_0+1}{p_0}}(B_\eta(\xi_{j,\epsilon}))}
+\Big\|\sum_{i=1}^\kappa V_i^{p_0}-W_2^{p_0}\Big\|_{L^{\frac{p_0+1}{p_0}}(\Omega\setminus \cup_{j=1}^\kappa B_\eta(\xi_{j,\epsilon}))}\Big).
\end{split}
\end{align}
\medskip

First, for the norm in the external area, if $p_0>\frac n{n-2}$
\begin{align}\label{4.14}
\begin{split}
&\Big\|\sum_{i=1}^\kappa U_i^{q_0}-W_1^{q_0}\Big\|_{L^{\frac{q_0+1}{q_0}}(\Omega\setminus \cup_{j=1}^\kappa B_\eta(\xi_{j,\epsilon}))}\\
\leq & C\Big(\sum_{l=1}^\kappa \int_{\Omega\setminus \cup_{j=1}^\kappa B_\eta(\xi_{j,\epsilon})} U_l^{q_0+1}\Big)^{\frac{q_0}{q_0+1}}=O\Big(\left(\frac{\delta}{\eta}\right)^{q_0(n-2)-\frac{nq_0}{q_0+1}}\Big)=O(\epsilon^{\frac12+\sigma})
\end{split}
\end{align}
and similarly,
\begin{align}\label{4.15}
\begin{split}
&\Big\|\sum_{i=1}^\kappa V_i^{p_0}-W_2^{p_0}\Big\|_{L^{\frac{p_0+1}{p_0}}(\Omega\setminus \cup_{j=1}^\kappa B_\eta(\xi_{j,\epsilon}))}=O\Big(\left(\frac{\delta}{\eta}\right)^{p_0(n-2)-\frac{np_0}{p_0+1}}\Big)
=O(\epsilon^{\frac12+\sigma}),
\end{split}
\end{align}
where we used the fact that $q_0,p_0>\frac n{n-2}$.

While if $p_0<\frac n{n-2}$, since $\frac{np_0q_0}{q_0+1}>\frac{(n-2)p_0-2}{2}$, we have
\begin{align}\label{4.14'}
\begin{split}
&\Big\|\sum_{i=1}^\kappa U_i^{q_0}-W_1^{q_0}\Big\|_{L^{\frac{q_0+1}{q_0}}(\Omega\setminus \cup_{j=1}^\kappa B_\eta(\xi_{j,\epsilon}))}\\
\leq & C\Big(\sum_{l=1}^\kappa \int_{\Omega\setminus \cup_{j=1}^\kappa B_\eta(\xi_{j,\epsilon})} U_l^{q_0+1}\Big)^{\frac{q_0}{q_0+1}}\\ =&O\Big(\left(\frac{\delta}{\eta}\right)^{q_0((n-2)p_0-2)-\frac{nq_0}{q_0+1}}\Big)
=O\Big(\left(\frac{\delta}{\eta}\right)^{\frac{np_0q_0}{q_0+1}}\Big)
=O(\epsilon^{\frac12+\sigma}).
\end{split}
\end{align}
Moreover, since $p_0>1,\frac{p_0}{p_0+1}>\frac12$, then
\begin{align}\label{4.15'}
\begin{split}
&\Big\|\sum_{i=1}^\kappa V_i^{p_0}-W_2^{p_0}\Big\|_{L^{\frac{p_0+1}{p_0}}(\Omega\setminus \cup_{j=1}^\kappa B_\eta(\xi_{j,\epsilon}))}\\&=O\Big(\left(\frac{\delta}{\eta}\right)^{p_0(n-2)-\frac{np_0}{p_0+1}}\Big)
=O\Big(\left(\frac{\delta}{\eta}\right)^{\frac{np_0}{q_0+1}}\Big)=O(\epsilon^{\frac12+\sigma}).
\end{split}
\end{align}
\medskip

Next,  for the norm in the internal area $B_\eta(\xi_{j,\epsilon})$ in \eqref{4.13}, we estimate subtly as follows.
\begin{align}\label{4.16}
\begin{split}
&\Big\|\sum_{i=1}^\kappa U_i^{q_0}-W_1^{q_0}\Big\|_{L^{\frac{q_0+1}{q_0}}(B_\eta(\xi_{j,\epsilon}))}+\Big\|\sum_{i=1}^\kappa V_i^{p_0}-W_2^{p_0}\Big\|_{L^{\frac{p_0+1}{p_0}}( B_\eta(\xi_{j,\epsilon}))}\\
\leq & \Big\|\sum_{i=1}^\kappa U_i^{q_0}-\Big(\sum_{i=1}^\kappa U_i\Big)^{q_0}\Big\|_{L^{\frac{q_0+1}{q_0}}( B_\eta(\xi_{j,\epsilon}))}+
\Big\|\Big(\sum_{i=1}^\kappa U_i\Big)^{q_0}-\Big(\sum_{i=1}^\kappa P U_i\Big)^{q_0}\Big\|_{L^{\frac{q_0+1}{q_0}}( B_\eta(\xi_{j,\epsilon}))}\\
&+\Big\|\sum_{i=1}^\kappa V_i^{p_0}-\Big(\sum_{i=1}^\kappa V_i\Big)^{p_0}\Big\|_{L^{\frac{p_0+1}{p_0}}( B_\eta(\xi_{j,\epsilon}))}
+\Big\|\Big(\sum_{i=1}^\kappa V_i\Big)^{p_0}-\Big(\sum_{i=1}^\kappa PV_i\Big)^{p_0}\Big\|_{L^{\frac{p_0+1}{p_0}}( B_\eta(\xi_{j,\epsilon}))}\\
:=&I_1+I_2+I_3+I_4.
\end{split}
\end{align}

For $I_1$, combining \eqref{4.14} or \eqref{4.14'}, it holds that
 \begin{align}\label{4.17}
\begin{split}
I_1= &\Big\|\sum_{i=1}^\kappa U_i^{q_0}-\Big(\sum_{i=1}^\kappa U_i\Big)^{q_0}\Big\|_{L^{\frac{q_0+1}{q_0}}( B_\eta(\xi_{j,\epsilon}))}\\
\leq& C\sum_{j=1}^\kappa\Big\| U_j^{q_0}-\Big(\sum_{i=1}^\kappa U_i\Big)^{q_0}\Big\|_{L^{\frac{q_0+1}{q_0}}( B_\eta(\xi_{j,\epsilon}))}
+C\sum_{j=1}^\kappa\sum_{i\neq j}^\kappa\| U_i^{q_0}\|_{L^{\frac{q_0+1}{q_0}}( B_\eta(\xi_{j,\epsilon}))}\\
\leq& C\sum_{j=1}^\kappa\sum_{i\neq j}^\kappa\|U_j^{q_0-1}U_i\|_{L^{\frac{q_0+1}{q_0}}( B_\eta(\xi_{j,\epsilon}))}+O(\epsilon^{\frac12+\sigma}).
\end{split}
\end{align}

We also discuss it in two cases. If $p_0>\frac n{n-2}$,  we take $\frac1{r_1}+\frac1{r_2}=\frac{q_0}{q_0+1}$ and estimate for $i\neq j$ that
 \begin{align}\label{4.18}
\begin{split}
 &\|U_j^{q_0-1}U_i\|_{L^{\frac{q_0+1}{q_0}}( B_\eta(\xi_{j,\epsilon}))}\\
\leq & C\|U_j^{q_0-1}\|_{L^{r_1}( B_\eta(\xi_{j,\epsilon}))}\|U_i\|_{L^{r_2}( B_\eta(\xi_{j,\epsilon}))}\\
\leq & C\|U_j^{q_0-1}\|_{L^{r_1}( B_\eta(\xi_{j,\epsilon}))}\frac{\delta_{i,\epsilon}^{n-2-\frac n{q_0+1}}}{\eta^{n-2-\frac{n}{r_2}}}=O\Big(\delta_j^{\frac{n}{r_1}-\frac{n(q_0-1)}{q_0+1}}\frac{\delta_{i,\epsilon}^{n-2-\frac n{q_0+1}}}{\eta_i^{n-2-\frac{n}{r_2}}}\Big)\\
=&O\Big(\epsilon^{\frac{n-1}{n-2}(n-2+\frac{n}{r_1}-\frac{nq_0}{q_0+1})-(n-2-\frac{n}{r_2})}\Big)
=O\Big(\epsilon^{1-\frac{n}{n-2}\frac{1}{r_2}}\Big)=O(\epsilon^{\frac12+\sigma}),
\end{split}
\end{align}
where we choose $r_2>\frac{2n}{n-2}$.

If $p_0<\frac n{n-2}$,  we take $\frac1{r_1}+\frac1{r_2}=\frac{q_0}{q_0+1}$ and estimate for $i\neq j$ that
 \begin{align}\label{4.18'}
\begin{split}
 &\|U_j^{q_0-1}U_i\|_{L^{\frac{q_0+1}{q_0}}( B_\eta(\xi_{j,\epsilon}))}\\
\leq & C\|U_j^{q_0-1}\|_{L^{r_1}( B_\eta(\xi_{j,\epsilon}))}\|U_i\|_{L^{r_2}( B_\eta(\xi_{j,\epsilon}))}\\
\leq & C\|U_j^{q_0-1}\|_{L^{r_1}( B_\eta(\xi_{j,\epsilon}))}\frac{\delta_{i,\epsilon}^{(n-2)p_0-2-\frac n{q_0+1}}}{\eta^{(n-2)p_0-2-\frac{n}{r_2}}}\\
 \leq & C\Big(\int_{ B_{\frac{\eta}{\delta}}(0)}\frac1{(1+|y|)^{((n-2)p_0-2)(q_0-1)r_1}}dy\Big)^{\frac1{r_1}}\frac{\delta_{i,\epsilon}^{(n-2)p_0-2-\frac n{q_0+1}}}{\eta^{(n-2)p_0-2-\frac{n}{r_2}}}\\
=&\begin{cases}
O\Big(\Big(\frac\delta\eta\Big)^{((n-2)p_0-2)q_0-\frac n{r_1}}\delta^{\frac{n}{r_1}-\frac{n(q_0-1)}{q_0+1}}\eta^{\frac n{r_2}}\Big)
=O\Big(\Big(\frac\delta\eta\Big)^{\frac{p_0q_0n}{q_0+1}}\Big)&if\ ((n-2)p_0-2)(q_0-1)r_1<n\\
O\Big(\Big(\frac\delta\eta\Big)^{\frac{(p_0+1)n}{q_0+1}-\frac n{r_2}-\theta_0}\Big)&if\ ((n-2)p_0-2)(q_0-1)r_1\geq n
\end{cases}\\
=&\begin{cases}
O(\epsilon^{\frac12+\sigma})&if\ ((n-2)p_0-2)(q_0-1)r_1<n\\
O\Big(\epsilon^{\frac{n}{(n-2)p_0-2}(\frac{p_0+1}{q_0+1}-\frac{1}{r_2})-\frac{\theta_0}{n-2}}\Big)=O(\epsilon^{\frac12+\sigma})&if\ ((n-2)p_0-2)(q_0-1)r_1\geq n
\end{cases},
\end{split}
\end{align}
where $\theta_0>0$ is any small constant and we choose $\frac n{r_2}<\frac{p_0+1}{q_0+1}n-\frac{(n-2)p_0-2}2$ and notice that $\frac{p_0q_0n}{q_0+1}>\frac{(n-2)p_0-2}2$.

Combining \eqref{4.17}, \eqref{4.18} and \eqref{4.18'},
 \begin{align}\label{4.19}
\begin{split}
I_1=O(\epsilon^{\frac12+\sigma}).
\end{split}
\end{align}

Similarly (actually more easily), we can obtain
 \begin{align}\label{4.20}
\begin{split}
I_3=O(\epsilon^{\frac12+\sigma}).
\end{split}
\end{align}
\medskip

Next, we estimate $I_2$ and $I_4$ in \eqref{4.16}.
We have that
\begin{align}\label{4.21}
\begin{split}
&\Big\|(\sum_{i=1}^\kappa U_i)^{q_0}-(\sum_{i=1}^\kappa P U_i)^{q_0}\Big\|_{L^{\frac{q_0+1}{q_0}}( B_\eta(\xi_{j,\epsilon}))}
%+\Big\|(\sum_{i=1}^\kappa V_i)^{p_0}-(\sum_{i=1}^\kappa PV_i)^{p_0}\Big\|_{L^{\frac{p_0+1}{p_0}}( B_\eta(\xi_{j,\epsilon}))}\\
\leq C\Big\|U_j^{q_0-1}\sum_{i=1}^\kappa (U_i-PU_i)\Big\|_{L^{\frac{q_0+1}{q_0}}( B_\eta(\xi_{j,\epsilon}))}\\
%+C\Big\|V_j^{p_0-1}\sum_{i=1}^\kappa (V_i-PV_i)\Big\|_{L^{\frac{p_0+1}{p_0}}( B_\eta(\xi_{j,\epsilon}))}\\
\leq & C \frac{\delta^{-\frac{n(q_0-1)}{q_0+1}+\frac{np_0}{q_0+1}}}{\eta^{\frac{p_0+1}{q_0+1}n}}\Big( \mathlarger\int_{B_\eta(\xi_{j,\epsilon})}
\sum_{i=1}^\kappa U\Big(\frac{x-\xi_{j,\epsilon}}{\delta_j}\Big)^{\frac{q_0^2-1}{q_0}}dx
\Big)^{\frac{q_0}{q_0+1}}\\
=&\displaystyle\begin{cases}
O\Big(\Big(\frac{\delta}\eta\Big)^{\frac{(p_0+1)n}{q_0+1}-\theta_0}
\Big),\ \ &(n-2)\frac{q_0^2-1}{q_0}\geq n\\
O\Big(\Big(\frac{\delta}\eta\Big)^{\frac{p_0q_0n}{q_0+1}}\Big),\ \ &(n-2)\frac{q_0^2-1}{q_0}< n
\end{cases}\\
=&O(\epsilon^{\frac12+\sigma}),
\end{split}
\end{align}
where we used the fact  that $\frac{(p_0+1)n}{q_0+1},\frac{p_0q_0n}{q_0+1}>\frac{(n-2)p_0-2}{2}$ and $\theta_0>0$ is any small constant.
\medskip

Similar estimate holds for $I_4$.

Combining \eqref{4.16}-\eqref{4.21}, we have
\begin{align}
\begin{split}
&\Big\|\sum_{i=1}^\kappa U_i^{q_0}-W_1^{q_0}\Big\|_{L^{\frac{q_0+1}{q_0}}(B_\eta(\xi_{j,\epsilon}))}+\Big\|\sum_{i=1}^\kappa V_i^{p_0}-W_2^{p_0}\Big\|_{L^{\frac{p_0+1}{p_0}}( B_\eta(\xi_{j,\epsilon}))}=O(\epsilon^{\frac12+\sigma}).
\end{split}
\end{align}

Finally, following \cite{pino-felmer-musso} and \cite{pino-felmer}, we can estimate $J$ in \eqref{4.12}.
\begin{align*}
\begin{split}
&J:=\Big\|W_2^p-W_2^{p_0}\Big\|_{L^{\frac{p_0+1}{p_0}}(\Omega)}+\Big\|W_1^q-W_1^{q_0}\Big\|_{L^{\frac{q_0+1}{q_0}}(\Omega)}           %%%%%%%%%%%%% 证明补充进来
=O(\epsilon |\ln\epsilon|)=O(\epsilon^{\frac12+\sigma}).
\end{split}
\end{align*}

\end{proof}

\medskip

\begin{Lem}
For any compact subsect $\Gamma_0$ of $\Gamma$, there exists $\epsilon_0>0$ and $\sigma>0$ such that for  $\epsilon\in(0,\epsilon_0)$
and $(\vec\xi,\vec\Lambda,\vec t)\in\Gamma_0$, there exists a unique $(\psi_\epsilon,\phi_\epsilon)\in E_{\xi,\Lambda,t}$ and $C>0$ such that \eqref{bot} holds and
\begin{align*}
\|(\psi_\epsilon,\phi_\epsilon)\|\leq C\epsilon^{\frac12+\sigma}.
\end{align*}

\end{Lem}

\begin{proof}
Recall that we reformulate \eqref{bot} as a fixed problem \eqref{fixed}
with
\begin{align*}
\begin{split}
N_\epsilon(\psi,\phi)=\Pi_{\xi,\Lambda,t}^\bot i^*(N_{1,\epsilon}(\phi),N_{2,\epsilon}(\psi)),
\end{split}
\end{align*}
with
\begin{align*}
\begin{split}
&N_{1,\epsilon}(\psi)=|W_1+\psi|^{q-1}(W_1+\psi)-W_1^q-q_0W_1^{q_0-1}\psi,\\
&N_{2,\epsilon}(\phi)=|W_2+\phi|^{p-1}(W_2+\phi)-W_2^p-p_0W_2^{p_0-1}\phi.
\end{split}
\end{align*}

It holds that
\begin{align}\label{4.24}
\begin{split}
&\|N_\epsilon(\psi,\phi)\| =\|\Pi_{\xi,\Lambda,t}^\bot i^*(N_{1,\epsilon}(\phi),N_{2,\epsilon}(\psi))\|\\
\leq & C\Big(\||W_1+\psi|^{q-1}(W_1+\psi)-W_1^q-qW_1^{q-1}\psi\|_{L^{\frac{q_0+1}{q_0}}(\Omega)}
+\||(q_0W_1^{q_0-1}-qW_1^{q-1})\psi\|_{L^{\frac{q_0+1}{q_0}}(\Omega)}\\
&+\||W_2+\phi|^{p-1}(W_2+\phi)-W_2^p-pW_2^{p-1}\phi\|_{L^{\frac{p_0+1}{p_0}}(\Omega)}+\||(p_0W_2^{p_0-1}-pW_2^{p-1})\phi\|_{L^{\frac{p_0+1}{p_0}}(\Omega)}
\Big)\\
:=& I_1+I_2+I_3+I_4.
\end{split}
\end{align}

For $I_1$, when $q>2$, noting that $q=q_0-\beta\epsilon$ and applying H\"older inequalities, we have
\begin{align}\label{4.25}
\begin{split}
I_1&\leq C\|W_1^{q-2}|\psi|^2+|\psi|^{q}\|_{L^{\frac{q_0+1}{q_0}}(\Omega)}
\\
&\leq C\Big(\|W_1\|_{L^{\frac{(q-2)(q_0+1)}{q_0-2}}(\Omega)}^{q-2}\|\psi\|_{L^{q_0+1}(\Omega)}^2
+\|\psi\|_{L^{\frac{q(q_0+1)}{q_0}}(\Omega)}^q
\Big)\\
&\leq  C\Big(\epsilon^{\frac{n((n-2)p_0-1)\beta}{((n-2)p_0-2)(q_0+1)}\epsilon}\|\psi\|_{L^{q_0+1}(\Omega)}^2
+|\Omega|^{\frac{\beta\epsilon}{q_0+1}}|\psi\|_{L^{q_0+1}(\Omega)}^q
\Big)\\
&\leq C\|\psi\|_{L^{q_0+1}(\Omega)}^2+C\|\psi\|_{L^{q_0+1}(\Omega)}^q.
\end{split}
\end{align}
While when $q\in(1,2)$, it holds that
\begin{align}\label{4.26}
\begin{split}
I_1&\leq C\||\psi|^{q}\|_{L^{\frac{q_0+1}{q_0}}(\Omega)}\leq C\|\psi\|_{L^{q_0+1}(\Omega)}^q.
\end{split}
\end{align}

For  $p=p_0-\alpha\epsilon$,
similar estimates gives  when $p>2$,
\begin{align}\label{4.27}
\begin{split}
I_3
&\leq C\|\phi\|_{L^{p_0+1}(\Omega)}^2+C\|\phi\|_{L^{p_0+1}(\Omega)}^p,
\end{split}
\end{align}
while when $p\in(1,2)$, we have
\begin{align}\label{4.28}
\begin{split}
I_3&\leq C\|\phi\|_{L^{p_0+1}(\Omega)}^p.
\end{split}
\end{align}

Next, using $q_0-q=\beta\epsilon,p_0-p=\alpha\epsilon$, we estimate
\begin{eqnarray*}
 I_2+I_4&=&\||(q_0W_1^{q_0-1}-qW_1^{q-1})\psi\|_{L^{\frac{q_0+1}{q_0}}(\Omega)}+\||(p_0W_2^{p_0-1}-pW_2^{p-1})\phi\|_{L^{\frac{p_0+1}{p_0}}(\Omega)}
\\
&\leq& \|\beta\epsilon W_1^{q_0-1}(1+q_0\log  W_1)\psi\|_{L^{\frac{q_0+1}{q_0}}(\Omega)}
+\|\alpha\epsilon W_2^{p_0-1}(1+p_0\log  W_2)\phi\|_{L^{\frac{p_0+1}{p_0}}(\Omega)}\\
&\leq& C(\epsilon+\epsilon\log\epsilon)\|U\|_{L^{q_0+1}(\Omega)}^{q_0-1}\|\psi\|_{L^{q_0+1}(\Omega)}+C\epsilon\Big\|U(\log U)^{\frac{q_0+1}{q_0-1}}\Big\|_{L^{q_0+1}(\Omega)}^{q_0-1}\|\psi\|_{L^{q_0+1}(\Omega)}\\
&&+C(\epsilon+\epsilon\log\epsilon)\|V\|_{L^{p_0+1}(\Omega)}^{p_0-1}\|\phi\|_{L^{p_0+1}(\Omega)}+C\epsilon\Big\|V(\log V)^{\frac{p_0+1}{p_0-1}}\Big\|_{L^{p_0+1}(\Omega)}^{p_0-1}\|\psi\|_{L^{p_0+1}(\Omega)}\\
&&+o(\epsilon)\|(\psi,\phi)\|\\
&=&O(\epsilon\log\epsilon)\|(\psi,\phi)\|.
\end{eqnarray*}

Using a standard argument, we prove that there exists some $C>0$ such that $T_{\xi,\Lambda,t}$ is a contradiction map on
$$M=\{(\psi,\phi)\in E_{\xi,\Lambda,t}:\|(\psi,\phi)\|\leq C\|R_\epsilon\|\}.$$
In view of Lemma \ref{lemerror}, there exists a unique solution $(\psi,\phi)\in E_{\xi,\Lambda,t}$ of \eqref{bot} satisfying $$\|(\psi,\phi)\|\leq C\|R_\epsilon\|\leq C\epsilon^{\frac12+\sigma}.$$
\end{proof}

\medskip

\section{The Reduced Problem}

Recall the energy functional
\begin{align}\label{I}
I_\epsilon(u,v):=\int_{\Omega}a(x)\nabla u\cdot\nabla vdx
-\frac1{p+1}\int_{\Omega}a(x)| v|^{p+1}dx-\frac1{q+1}\int_{\Omega}a(x)|u|^{q+1}dx.
\end{align}
It is well-known that $(v_1,v_2)\in X$ is a solution to \eqref{eqv} if and only if it is a positive solution of $I_\epsilon$.
Set the reduced energy
\begin{align}\label{J}
J_\epsilon(\vec\xi,\vec\Lambda,\vec t)=I_\epsilon\big(W_1+\psi_\epsilon,W_2+\phi_\epsilon\big),
\end{align}
where
$$W_1=\sum_{i=1}^\kappa PU_{i}=\sum_{i=1}^\kappa PU_{\xi_{i,\epsilon},\delta_{i,\epsilon}},\ \ \  W_2=\sum_{i=1}^\kappa PV_{i}=\sum_{i=1}^\kappa PV_{\xi_{i,\epsilon},\delta_{i,\epsilon}},$$
and $(\psi_\epsilon,\phi_\epsilon)$ is a solution to \eqref{bot} for given $(\vec\xi,\vec\Lambda,\vec t)\in \Gamma$ found in
Corollary \ref{corlinear}.

It is standard to give that
\begin{Prop}\label{prop5.1}
The function pair $\big(W_1+\psi_\epsilon,W_2+\phi_\epsilon\big)$ is a critical point of $I_\epsilon$ if and only if the points $(\vec\xi,\vec\Lambda,\vec t)$ is a critical point of $J_\epsilon$.

\end{Prop}

\medskip

We are reduced to find the critical points of $J_\epsilon$. For this purpose, we give the asymptotic expansion as follows.
\begin{Prop}\label{prop5.2}
It holds that there exist constants $c_i,i=1,2,\ldots,6$ with $c_4,c_5,c_5',c_6>0$ such that
\begin{align}
\begin{split}
&J_\epsilon(\vec\xi,\vec\Lambda,\vec t)=(c_1+c_2\epsilon\log\epsilon)\sum_{i=1}^\kappa a(\xi_{i})+\epsilon\sum_{i=1}^\kappa \Big[c_3a(\xi_{i})+c_4\langle\nabla a(\xi_{i}),\gamma(\xi_{i})\rangle t_i\\
&\quad+\begin{cases} c_5a(\xi_{i})
\Big(\frac{\Lambda_i}{2t_i}\Big)^{n-2}\  &if\ p_0>\frac n{n-2}\\
 c_5'a(\xi_{i})\Big(\frac{\Lambda_i}{2t_i}\Big)^{(n-2)p_0-2}\  &if\ p_0<\frac n{n-2}
\end{cases}-c_6 a(\xi_{i})\log \Lambda_i\Big]+O(\epsilon^{1+\sigma}).
\end{split}
\end{align}

\end{Prop}

\begin{proof}
Since
\begin{align*}
\Big\langle \big(I'_u(W_1+\varphi_1,W_2+\varphi_2),I'_v(W_1+\varphi_1,W_2+\varphi_2)\big),(\varphi_1,\varphi_2)\Big\rangle=0, \ \forall\, (\varphi_1,\varphi_2)\in E_{\xi,\Lambda,t},
\end{align*}
 there are $t,s\in(0,1)$ such that
\begin{align*}
 &J_\epsilon(\vec\xi,\vec\Lambda,\vec t) \\
 =&I_\epsilon(W_1,W_2)-\frac12\langle D^2I(W_1+t\psi,W_2+s\phi)(\psi,\phi),(\psi,\phi)\rangle\\
 =&I_\epsilon(W_1,W_2)-\frac12\int_{\R^n}a(x)\Big(2\nabla\psi\cdot\nabla\phi-q (W_1+t\psi)^{q-1}\psi^2-p(W_2+s\phi)^{p-1}\phi^2\Big)\\
 =&I_\epsilon(W_1,W_2)+\frac{1}2\int_{\R^n}a(x)\Big(q ((W_1+t\psi)^{q-1}-W_1^{q-1})\psi^2-(N_2(\psi)+R_{2,\epsilon})\psi^2\\&
\hspace{4.2cm}+
p ((W_2+s\phi)^{p-1}-W_2^{p-1})\phi^2
-(N_1(\phi)+R_{1,\epsilon})\phi^2\Big).
\end{align*}

Note that
\begin{align*}
&\int_{\R^n}a(x)\Big((N_2(\psi)+R_{2,\epsilon})\psi^2+(N_1(\phi)+R_{1,\epsilon})\phi^2\Big)=O(\epsilon^{1+\sigma}).
\end{align*}
Therefore, we obtain from Lemma \ref{lema4} and Lemma \ref{lema5} that
\begin{align*}
\begin{split}
&J_\epsilon(\vec\xi,\vec\Lambda,\vec t)=I_\epsilon(W_1,W_2)+O(\epsilon^{1+\sigma})\\
=& \frac {2A_1}n\sum_{i=1}^\kappa\Big(a(\xi_{i})+\langle\nabla a(\xi_{i}),\gamma(\xi_{i})\rangle t_i\epsilon \Big)\\&+\begin{cases} \frac{b_{n,p_0}B_2}{\gamma_n}\sum_{i=1}^\kappa a(\xi_{i})\epsilon\Big(\frac{\Lambda_i}{2t_i}\Big)^{n-2} &if\  p_0>\frac n{n-2}\\
  \frac{b_{n,p_0}\mathcal I_i}{\gamma_n}\sum_{i=1}^\kappa a(\xi_{i})\epsilon\Big(\frac{\Lambda_i}{2t_i}\Big)^{(n-2)p_0-2} &if\  p_0<\frac n{n-2}
  \end{cases}\\&-\epsilon\log\epsilon\frac{n(n-1)}{n-2}\Big(\frac{A_1}{(q_0+1)^2}+\frac{B_1}{(p_0+1)^2}\Big)\sum_{j=1}^\kappa a(\xi_{i})\\
&-\epsilon\Big(\frac{\beta A_1}{(q_0+1)^2}+\frac{\alpha B_1}{(p_0+1)^2}\Big)\sum_{j=1}^\kappa a(\xi_{i})
-\epsilon\Big(\frac{nA_1}{(q_0+1)^2}+\frac{nB_1}{(p_0+1)^2}\Big)\sum_{j=1}^\kappa a(\xi_{i})\log \Lambda_i\\&
+\epsilon\Big(\frac{A_3}{q_0+1}+\frac{B_3}{p_0+1}\Big)\sum_{j=1}^\kappa a(\xi_{i})+O(\epsilon^{1+\sigma})
\\=& \Big[\frac {2A_1}n-\epsilon\log\epsilon\frac{n(n-1)}{n-2}\Big(\frac{A_1}{(q_0+1)^2}+\frac{B_1}{(p_0+1)^2}\Big)\Big]\sum_{i=1}^\kappa a(\xi_{i})\\&+
\epsilon\sum_{i=1}^\kappa \Big[-A_1\Big(\frac{\beta}{(q_0+1)^2}+\frac{\alpha}{(p_0+1)^2}\Big)a(\xi_{i})+\Big(\frac{A_3}{q_0+1}+\frac{B_3}{p_0+1}\Big) a(\xi_{i})\\
&+\frac {2A_1}n\langle\nabla a(\xi_{i}),\gamma(\xi_{i})\rangle t_i-\Big(\frac{nA_1}{(q_0+1)^2}+\frac{nB_1}{(p_0+1)^2}\Big) a(\xi_{i})\log \Lambda_i\\
&+\begin{cases} \frac{b_{n,p_0}B_2}{\gamma_n}\sum_{i=1}^\kappa a(\xi_{i})\epsilon\Big(\frac{\Lambda_i}{2t_i}\Big)^{n-2} &if\  p_0>\frac n{n-2}\\
  \frac{b_{n,p_0}\mathcal I_i}{\gamma_n}\sum_{i=1}^\kappa a(\xi_{i})\epsilon\Big(\frac{\Lambda_i}{2t_i}\Big)^{(n-2)p_0-2} &if\  p_0<\frac n{n-2}
  \end{cases}
+O(\epsilon^{\sigma})\Big].
\end{split}
\end{align*}

\end{proof}

\medskip

\begin{proof}[
\textbf{Proof of Theorem \ref{th1}}]
From Proposition \ref{prop5.2}, there exist $c_1$ and $c_2$ such that
\begin{align*}
J_\epsilon(\vec\xi,\vec\Lambda,\vec t)=(c_1+c_2\epsilon\log\epsilon)\sum_{i=1}^\kappa a(\xi_{i})+O(\epsilon),
\end{align*}
where $\left|\frac{O(\epsilon)}{\epsilon}\right|\leq C$ uniformly on compact sets of $\Gamma$. Since $\tilde\xi_{i},i=1,\ldots,\kappa$ are non-degenerate critical points of $a$
constrained to $\partial\Omega$, then there exist $\vec \xi^{(\epsilon)}=(\xi_1^{(\epsilon)},\ldots,\xi_\kappa^{(\epsilon)})$ such that $\xi_{i}^{(\epsilon)}\rightarrow\tilde\xi_{i}$
as $\epsilon\rightarrow0$, and $\nabla_{\vec \xi}\tilde J_\epsilon(\vec \xi^{(\epsilon)},\vec\Lambda,\vec t)=0$.

Moreover, by Proposition \ref{prop5.2},
\begin{align*}
\begin{split}
&J_\epsilon(\vec \xi^{(\epsilon)},\vec\Lambda,\vec t)-(c_1+c_2\epsilon\log\epsilon)\sum_{i=1}^\kappa a(\xi_{i}^{(\epsilon)})\\
=&\epsilon\sum_{i=1}^\kappa \Big[c_3a(\xi_{i}^{(\epsilon)})+c_4\langle\nabla a(\xi_{i}^{(\epsilon)}),\gamma(\xi_{i}^{(\epsilon)})\rangle t_i+\begin{cases} c_5a(\xi_{i}^{(\epsilon)})
\Big(\frac{\Lambda_i}{2t_i}\Big)^{n-2}\  &if\ p_0>\frac n{n-2}\\
 c_5'a(\xi_{i}^{(\epsilon)})\Big(\frac{\Lambda_i}{2t_i}\Big)^{(n-2)p_0-2}\  &if\ p_0<\frac n{n-2}
\end{cases}\\&-c_6 a(\xi_{i}^{(\epsilon)})\log \Lambda_i\Big]
+O(\epsilon^{1+\sigma})\\
=&\epsilon\sum_{i=1}^\kappa \Big[c_3a(\tilde\xi_{i})+c_4\langle\nabla a(\tilde\xi_{i}),\gamma(\tilde\xi_{i})\rangle t_i+\begin{cases} c_5a(\tilde\xi_{i})
\Big(\frac{\Lambda_i}{2t_i}\Big)^{n-2}\  &if\ p_0>\frac n{n-2}\\
 c_5'a(\tilde\xi_{i})\Big(\frac{\Lambda_i}{2t_i}\Big)^{(n-2)p_0-2}\  &if\ p_0<\frac n{n-2}
\end{cases}\\&-c_6 a(\tilde\xi_{i})\log \Lambda_i\Big]
+O(\epsilon^{1+\sigma}).
\end{split}
\end{align*}
Observing that
\begin{align*}
\begin{split}
(\vec\Lambda,\vec t)\rightarrow \sum_{i=1}^\kappa \Big[c_4\langle\nabla a(\tilde\xi_{i}),\gamma(\tilde\xi_{i})\rangle t_i+c_5a(\tilde\xi_{i})
\Big(\frac{\Lambda_i}{2t_i}\Big)^{n-2}-c_6 a(\tilde\xi_{i})\log \Lambda_i\Big]
\end{split}
\end{align*}
has a minimum point which is stable up to $C^0$-perturbations, we can check that  there exists $(\vec \Lambda_\epsilon,\vec t_\epsilon)$ such that
 $\nabla_{\vec\Lambda,\vec t}\tilde J_\epsilon(\vec \xi^{(\epsilon)},\vec \Lambda_\epsilon,\vec t_\epsilon)=0$. Therefore, $\tilde J_\epsilon$ has a critical point.

\end{proof}

\medskip

\section*{Appendix}

\appendix

\section{Energy expansion}
\renewcommand{\theequation}{A.\arabic{equation}}

Note that when $n\geq3$ and $p_0>\frac n{n-2}$, the following positive quantities are well-defined:
\begin{align}
\begin{split}
&A_1=\int_{\R^n}U^{q_0+1},\ \ A_2=\int_{\R^n}U^{q_0},\ \ A_3=\int_{\R^n}U^{q_0+1} \log U,\\
&B_1=\int_{\R^n}V^{p_0+1}=A_1,\ \ B_2=\int_{\R^n}V^{p_0},\ \ B_3=\int_{\R^n}V^{p_0+1} \log V.
\end{split}
\end{align}
Moreover, if $p_0<\frac n{n-2}$, since $((n-2)p_0-2)q_0=\frac{q_0(p_0+1)n}{q_0+1}>n,((n-2)p_0-2)(q_0+1)>n$ and $(q_0+1)(n-2)>n$, then
$A_i(i=1,2,3),B_i(i=1,3)$ are all well-defined as well.

\medskip
Recall the numbers $a_{n,p}$ and $b_{n,p}$ appeared in Lemma \ref{lemasym}.
We define the main term of $I_0$
\begin{align}\label{J0}
 I_0(u,v)=\int_{\Omega}a(x)\nabla u\cdot\nabla v
-\frac1{p_0+1}\int_{\Omega}a(x)|v|^{p_0+1}dx-\frac1{q_0+1}\int_{\Omega}a(x)|u|^{q_0+1}dx.
\end{align}
Correspondingly,
\begin{align}\label{J0}
\begin{split}
\tilde J_0:&= I_0(W_1,W_2)\\
&=\int_{\Omega}a(x)\nabla W_1\cdot\nabla W_2
-\frac1{p_0+1}\int_{\Omega}a(x)|W_2|^{p_0+1}dx-\frac1{q_0+1}\int_{\Omega}a(x)|W_1|^{q_0+1}dx.
\end{split}
\end{align}

We start with some key estimates.

\begin{Lem}\label{lema1}
For $i=1,\dots,\kappa$, there holds that
\begin{align*}
&\int_{B_\eta(\xi_{i})}a(x)U_i^{q_0+1}dx=A_1\left(a(\xi_{i})+\langle\nabla a(\xi_{i}),\gamma(\xi_{i})\rangle t_i\epsilon\right)+O(\epsilon^{1+\sigma}),\\
&\int_{B_\eta(\xi_{i})}a(x)V_i^{p_0+1}dx=B_1\left(a(\xi_{i})+\langle\nabla a(\xi_{i}),\gamma(\xi_{i})\rangle t_i\epsilon\right)+O(\epsilon^{1+\sigma}).
\end{align*}
\end{Lem}

\begin{proof}

For $i=1,\dots,\kappa$, note that $\xi_{i,\epsilon}=\xi_{i}+\eta_i\nu(\xi_{i})$.

First for $p_0>\frac n{n-2}$, there exists small $\sigma\leq p_0-\frac n{n-2}<q_0-\frac n{n-2}$ such that
\begin{align}\label{aa}
\begin{split}
&\int_{B_\eta(\xi_{i,\epsilon})}a(x)U_i^{q_0+1}dx=\int_{B_{\frac{\eta}{\delta_{i,\epsilon}}}(0)}a(\delta_{i,\epsilon} y+\xi_{i}+\eta_i\gamma(\xi_{i}))U^{q_0+1}(y)dy\\
=&\int_{B_{\frac{\eta}{\delta_{i,\epsilon}}}(0)}a(\xi_{i})U^{q_0+1}(y)dy+\int_{B_{\frac{\eta}{\delta_{i,\epsilon}}}(0)}(a(\delta_{i,\epsilon} y+\xi_{i}+\eta_i\gamma(\xi_{i}))-a(\xi_{i}))U^{q_0+1}(y)dy\\
=&A_1a(\xi_{i})+O\Big(\int_{\frac{\eta}{\delta_{i,\epsilon}}}^\infty\frac1{(1+|y|^{n-2})^{q_0+1}}dy\Big)\\
&+\int_{B_{\frac{\eta}{\delta_{i,\epsilon}}}(0)}(\langle\nabla a(\xi_{i}),\gamma(\xi_{i})\eta_i\rangle+\delta_{i,\epsilon}\langle\nabla a(\xi_{i}),y\rangle+h(y))U^{q_0+1}(y)dy\\
=&A_1a(\xi_{i})+O(\epsilon^{1+\sigma})+A_1\langle\nabla a(\xi_{i}),\gamma(\xi_{i})\eta_i\rangle+O(\eta_i^2),
\end{split}\end{align}
where $|h(y)|\leq c(\delta_{i,\epsilon}^2|y|^2+\delta_{i,\epsilon}\eta_i|y|+\eta_i^2))$.

Next in the case of $p_0<\frac n{n-2}$,
\begin{align*}
&\int_{B_\eta(\xi_{i,\epsilon})}a(x)U_i^{q_0+1}dx=\int_{B_{\frac{\eta}{\delta_{i,\epsilon}}}(0)}a(\delta_{i,\epsilon} y+\xi_{i}+\eta_i\gamma(\xi_{i}))U^{q_0+1}(y)dy\\
=&A_1a(\xi_{i})+O\Big(\int_{\frac{\eta}{\delta_{i,\epsilon}}}^\infty\frac1{(1+|y|^{(n-2)p_0-2})^{q_0+1}}dy\Big)\\
&+\int_{B_{\frac{\eta}{\delta_{i,\epsilon}}}(0)}(\langle\nabla a(\xi_{i}),\gamma(\xi_{i})\eta_i\rangle+\delta_{i,\epsilon}\langle\nabla a(\xi_{i}),y\rangle+g(y))U^{q_0+1}(y)dy\\
=&A_1a(\xi_{i})+O(\epsilon^{1+\sigma})+A_1\langle\nabla a(\xi_{i}),\gamma(\xi_{i})\eta_i\rangle+O\Big(\epsilon^{1+\sigma}\Big),
\end{align*}
where   $|g(y)|\leq c(\delta_{i,\epsilon}^2|y|^2+\delta_{i,\epsilon}\eta_i|y|+\eta_i^2))$, and we have used the fact that
\begin{align*}
&\int_{\frac{\eta}{\delta_{i,\epsilon}}}^\infty\frac1{(1+|y|^{(n-2)p_0-2})^{q_0+1}}dy
=O\Big(\Big(\frac\delta\eta\Big)^{p_0n}\Big)=O\Big(\epsilon^{\frac{p_0n}{(n-2)p_0-2}}\Big)=O\Big(\epsilon^{1+\sigma}\Big).
\end{align*}

Similar estimate as in \eqref{aa}, it holds that
\begin{align*}
&\int_{B_\eta(\xi_{i,\epsilon})}a(x)V_i^{p_0+1}dx=B_1(a(\xi_{i})+\langle\nabla a(\xi_{i}),\gamma(\xi_{i})\eta_i\rangle)+O(\epsilon^{1+\sigma}).
\end{align*}

\end{proof}

\begin{Lem}\label{lema2}If $p_0>\frac n{n-2}$, then
for $i=1,\dots,\kappa$  there holds  that
\begin{align*}
%&\int_{B_\eta(\xi_{i})}a(x)U_i^{q_0}(PU_i-U_i)dx=-\frac{a_{n,p_0}A_2}{\gamma_n}a(\xi_{i})\epsilon^{\frac{n}{n-2}\frac{p_0+1}{q_0+1}}
%\Big(\frac{\Lambda_i}{2t_i}\Big)^{\frac{p_0+1}{q_0+1}n}+O(\epsilon^{\frac{n}{n-2}\frac{p_0+1}{q_0+1}+\sigma}),\\
&\int_{B_\eta(\xi_{i,\epsilon})}a(x)V_i^{p_0}(PV_i-V_i)dx=-\frac{b_{n,p_0}B_2}{\gamma_n}a(\xi_{i})\epsilon\Big(\frac{\Lambda_i}{2t_i}\Big)^{n-2}+O(\epsilon^{1+\sigma}).
\end{align*}
\end{Lem}

\begin{proof}
Using Lemma \ref{lemH} and Lemma \ref{lemP}, we have that
\begin{align*}
&\int_{B_\eta(\xi_{i,\epsilon})}a(x)V_i^{p_0}(PV_i-V_i)dx\\
=&\int_{B_\eta(\xi_{i,\epsilon})}a(x)V_i^{p_0}\Big(-\frac{b_{n,p_0}}{\gamma_n}\delta^{\frac{n}{q_0+1}}H(x,\xi_{i,\epsilon})+R_{2,\delta_{i,\epsilon},\xi_{i,\epsilon}}(x)\Big)dx\\
=&-\frac{b_{n,p_0}}{\gamma_n}\int_{B_{\frac{\eta}{\delta_{i,\epsilon}}}(0)}\delta_{i,\epsilon}^{n-2}a(\delta_{i,\epsilon}y+\xi_{i,\epsilon})
H(\delta_{i,\epsilon}y+\xi_{i,\epsilon},\xi_{i,\epsilon})V^{p_0}(y)dy
+O\left(\frac{\delta_{i,\epsilon}^{\frac n{q_0+1}+1}}{\eta_i^{n-1}}\int_{B_\eta(\xi_{i,\epsilon})}V_i^{p_0}\right)\\
=&-\frac{b_{n,p_0}}{\gamma_n}\int_{B_{\frac{\eta}{\delta_{i,\epsilon}}}(0)}\delta_{i,\epsilon}^{n-2}\frac{a(\delta_{i,\epsilon}y+\xi_{i,\epsilon})}{|\delta_{i,\epsilon}y+\xi_{i,\epsilon}-\bar\xi_{i,\epsilon}|^{n-2}}
V^{p_0}(y)dy
+O\Big(\frac{\delta_{i,\epsilon}^{n-1}}{\eta_i^{n-1}}\Big)\\
=&-\frac{b_{n,p_0}}{\gamma_n}a(\xi_{i})B_2\Big(\frac{\delta_{i,\epsilon}}{2\eta_i}\Big)^{n-2}+O\Big(\frac{\delta_{i,\epsilon}^{n-1}}{\eta_i^{n-1}}\Big)\\
&-\frac{b_{n,p_0}}{\gamma_n}\Big(\frac{\delta_{i,\epsilon}}{2\eta_i}\Big)^{n-2}
\int_{B_{\frac{\eta}{\delta_{i,\epsilon}}}(0)}\Big(\langle\nabla a(\xi_{i}),\gamma(\xi_{i})\rangle\eta_i+\delta_{i,\epsilon}\langle\nabla a(\xi_{i}),y\rangle+h(y)\Big)V^{p_0}(y)dy\\
=&-\frac{b_{n,p_0}}{\gamma_n}a(\xi_{i})B_2\Big(\frac{\delta_{i,\epsilon}}{2\eta_i}\Big)^{n-2}+O(\epsilon^{1+\sigma})
=-\frac{b_{n,p_0}}{\gamma_n}a(\xi_{i})B_2\epsilon\Big(\frac{\Lambda_i}{2t_i}\Big)^{n-2}+O(\epsilon^{1+\sigma}).
\end{align*}
\iffalse
Similar estimates gives

\fi
We conclude the proof.

\end{proof}

\begin{Lem}\label{lema2'}If $p_0<\frac n{n-2}$,
for $i=1,\dots,\kappa$, we have
\begin{align*}
%&\int_{B_\eta(\xi_{i})}a(x)U_i^{q_0}(PU_i-U_i)dx=-\frac{a_{n,p_0}A_2}{\gamma_n}a(\xi_{i})\epsilon^{\frac{n}{n-2}\frac{p_0+1}{q_0+1}}
%\Big(\frac{\Lambda_i}{2t_i}\Big)^{\frac{p_0+1}{q_0+1}n}+O(\epsilon^{\frac{n}{n-2}\frac{p_0+1}{q_0+1}+\sigma}),\\
&\int_{B_\eta(\xi_{i,\epsilon})}a(x)V_i^{p_0}(PV_i-V_i)dx=-\frac{b_{n,p_0}\mathcal I_i}{\gamma_n}a(\xi_{i})\epsilon\Big(\frac{\Lambda_i}{2t_i}\Big)^{(n-2)p_0-2}+O(\epsilon^{1+\sigma}),
\end{align*}
where $\mathcal I_i>0$ is some constant.
\end{Lem}

\begin{proof}
Using Lemma \ref{lemH} and Lemma \ref{lemP'2}, we have that
\begin{align*}
&\int_{B_\eta(\xi_{i,\epsilon})}a(x)V_i^{p_0}(PV_i-V_i)dx\\
=&\int_{B_\eta(\xi_{i,\epsilon})}a(x)V_i^{p_0}\Big(-\frac{b_{n,p_0}}{\gamma_n}\delta^{\frac{n}{q_0+1}}H(x,\xi_{i,\epsilon})+R_{2,\delta_{i,\epsilon},\xi_{i,\epsilon}}(x)\Big)dx\\
%=&-\frac{b_{n,p_0}}{\gamma_n}\int_{B_{\frac{\eta}{\delta_{i,\epsilon}}}(0)}\delta_{i,\epsilon}^{n-2}a(\delta_{i,\epsilon}y+\xi_{i,\epsilon})H(\delta_{i,\epsilon}y+\xi_{i,\epsilon},\xi_{i,\epsilon})V^{p_0}(y)dy
%+O\left(\frac{\delta_{i,\epsilon}^{\frac n{q_0+1}+1}}{\eta_i^{n-1}}\int_{B_\eta(\xi_{i,\epsilon})}V_i^{p_0}\right)\\
=&-\frac{b_{n,p_0}}{\gamma_n}\int_{B_{\frac{\eta}{\delta_{i,\epsilon}}}(0)}\delta_{i,\epsilon}^{n-2}\frac{a(\delta_{i,\epsilon}y+\xi_{i,\epsilon})}{|\delta_{i,\epsilon}y+\xi_{i,\epsilon}-\bar\xi_{i,\epsilon}|^{n-2}}
V^{p_0}(y)dy
+O\Big(\frac{\delta_{i,\epsilon}^{(n-2)p_0-1}}{\eta_i^{(n-2)p_0-1}}\Big)\\
=&-\frac{b_{n,p_0}}{\gamma_n}a(\xi_{i})\frac{\delta^{n-2}_{i,\epsilon}}{{2\eta_i}^{(n-2)p_0-2}}
\int_{B_{\frac{\eta}{\delta_{i,\epsilon}}}(0)} \frac{V^{p_0}(y)}{|\delta_{i,\epsilon}y+\xi_{i,\epsilon}-\bar\xi_{i,\epsilon}|^{n-(n-2)p_0+\sigma_0}}
dy\\
&-\frac{b_{n,p_0}}{\gamma_n}\frac{\delta^{n-2}_{i,\epsilon}}{{2\eta_i}^{(n-2)p_0-2}}
\int_{B_{\frac{\eta}{\delta_{i,\epsilon}}}(0)}\Big(\langle\nabla a(\xi_{i}),\gamma(\xi_{i})\rangle\eta_i+\delta_{i,\epsilon}\langle\nabla a(\xi_{i}),y\rangle+h(y)\Big)\\
&\qquad \cdot\frac{V^{p_0}(y)}{|\delta_{i,\epsilon}y+\xi_{i,\epsilon}-\bar\xi_{i,\epsilon}|^{n-(n-2)p_0}}dy+O\Big(\frac{\delta_{i,\epsilon}^{(n-2)p_0-1}}{\eta_i^{(n-2)p_0-1}}\Big)\\
=&-\frac{b_{n,p_0}}{\gamma_n}a(\xi_{i}) \Big(\frac{\delta_{i,\epsilon}}{2\eta_i}\Big)^{(n-2)p_0-2}\mathcal I_i+O(\epsilon^{1+\sigma})\\&
=-\frac{b_{n,p_0}}{\gamma_n}a(\xi_{i})\epsilon\Big(\frac{\Lambda_i}{2t_i}\Big)^{(n-2)p_0-2}\mathcal I_i+O(\epsilon^{1+\sigma}),
\end{align*}
where, by setting $x_i=\frac{\xi_{i,\epsilon}-\bar\xi_{i,\epsilon}}{\delta_{i,\epsilon}}$, we denote
\begin{align*}
\mathcal I_i:&=\int_{B_{\frac{\eta}{\delta_{i,\epsilon}}}(0)} \frac{\delta_{i,\epsilon}^{n-(n-2)p_0}V^{p_0}(y)}{|\delta_{i,\epsilon}y+\xi_{i,\epsilon}-\bar\xi_{i,\epsilon}|^{n-(n-2)p_0}}
dy\\&=O\Big(\int_{B_{\frac{\eta}{\delta_{i,\epsilon}}}(0)} \frac{1}{\Big|y+\frac{\xi_{i,\epsilon}-\bar\xi_{i,\epsilon}}{\delta_{i,\epsilon}}\Big|^{n-(n-2)p_0}(1+|y|)^{(n-2)p_0}}
dy\Big)\\
&=O\Big(\int_{B_{\frac{\eta}{\delta_{i,\epsilon}}}(0)} \frac{1}{|y+x_i|^{n-(n-2)p_0}(1+|y|)^{(n-2)p_0}}
dy\Big)\\
&=O\Big( \int_{B_{\frac12}(0)} \frac{1}{|z+\frac{x_i}{|x_i|}|^{n-(n-2)p_0}(\frac1{|x_i|}+|z|)^{(n-2)p_0}}
dz\Big)=O(1).
\end{align*}
%since $|x_i|=\frac{2\eta_i}{\delta_i}\rightarrow+\infty$.

We conclude the proof.

\end{proof}

\begin{Lem}\label{lema3}
For $i,j=1,\dots,\kappa, i\neq j$, there holds that
\begin{align*}
%&\int_{B_\eta(\xi_{i,\epsilon})}a(x)U_i^{q_0}PU_jdx=O(\epsilon^{1+\sigma}),\ \ \
\int_{B_\eta(\xi_{i,\epsilon})}a(x)V_i^{p_0}PV_jdx=O(\epsilon^{1+\sigma}).
\end{align*}
\end{Lem}

\begin{proof}
By Lemma \ref{lemH} and Lemma \ref{lemP}, we have for $p_0>\frac n{n-2}$ that
\begin{align*}
&\int_{B_\eta(\xi_{i,\epsilon})}a(x)V_i^{p_0}PV_jdx\\
&=\int_{B_\eta(\xi_{i,\epsilon})}a(x)V_i^{p_0}(x)\Big(V_j(x)-\frac{b_{n,p_0}}{\gamma_n}\delta^{\frac{n}{q_0+1}}H(x,\xi_{j,\epsilon})+R_{2,\delta_j,\xi_{i,\epsilon}}(x)\Big)dx\\
&=\delta_{i,\epsilon}^{\frac n{q_0+1}}\delta_{j,\epsilon}^{\frac n{p_0+1}}\int_{B_{\frac{\eta}{\delta_{i,\epsilon}}}(0)}a(\delta_{i,\epsilon}y+\xi_{i,\epsilon})V^{p_0}(y)\\
&\qquad\times\Big(\delta_j^{-\frac n{p_0+1}-\frac n{q_0+1}}V\Big(\frac{|\delta_{i,\epsilon}y+\xi_{i,\epsilon}-\xi_{j,\epsilon}|}{\delta_j}\Big)
-\frac{b_{n,p_0}}{\gamma_n}\frac{1}{|\delta_{i,\epsilon}y+\xi_{i,\epsilon}-\bar\xi_{j,\epsilon}|^{n-2}}\Big)dy+O(\epsilon^{1+\sigma})\\
&=O(\delta^{n-2})+O(\epsilon^{1+\sigma})=O(\epsilon^{1+\sigma}).
\end{align*}

While by Lemma \ref{lemP'2}, for $p_0<\frac n{n-2}$, similar as the proof of Lemma \ref{lema2'}, we can prove that
\begin{align*}
&\int_{B_\eta(\xi_{i,\epsilon})}a(x)V_i^{p_0}PV_jdx%\\
%&=\int_{B_\eta(\xi_{i,\epsilon})}a(x)V_i^{p_0}(x)\Big(V_j(x)-\frac{b_{n,p_0}}{\gamma_n}\delta^{\frac{n}{q_0+1}}H(x,\xi_{j,\epsilon})+R_{2,\delta_j,\xi_{i,\epsilon}}(x)\Big)dx\\
%&=\delta_{i,\epsilon}^{\frac n{q_0+1}}\delta_{j,\epsilon}^{\frac n{p_0+1}}\int_{B_{\frac{\eta}{\delta_{i,\epsilon}}}(0)}a(\delta_{i,\epsilon}y+\xi_{i,\epsilon})V^{p_0}(y)\\
%&\qquad\times\Big(\delta_j^{-\frac n{p_0+1}-\frac n{q_0+1}}V\Big(\frac{|\delta_{i,\epsilon}y+\xi_{i,\epsilon}-\xi_{j,\epsilon}|}{\delta_j}\Big)-\frac{b_{n,p_0}}{\gamma_n}\frac{1}{|\delta_{i,\epsilon}y+\xi_{i,\epsilon}-\bar\xi_{j,\epsilon}|^{n-2}}\Big)dy+O(\epsilon^{1+\sigma})\\
 =O(\delta^{(n-2)p_0-2})+O(\epsilon^{1+\sigma})=O(\epsilon^{1+\sigma}).
\end{align*}

\end{proof}
\medskip

Now we are in a position to show the Energy expansion of the main term.

\begin{Lem}\label{lema4}
There holds that
\begin{align}\label{main}
\begin{split}
  \tilde J_0 &= I_0(W_1,W_2) = \frac {2A_1}n\sum_{i=1}^\kappa\Big(a(\xi_{i})+\langle\nabla a(\xi_{i}),\gamma(\xi_{i})\rangle t_i\epsilon \Big)\\
  &+\begin{cases} \frac{b_{n,p_0}B_2}{\gamma_n}\sum_{i=1}^\kappa a(\xi_{i})\epsilon\Big(\frac{\Lambda_i}{2t_i}\Big)^{n-2} &if\  p_0>\frac n{n-2}\\
  \frac{b_{n,p_0}\mathcal I_i}{\gamma_n}\sum_{i=1}^\kappa a(\xi_{i})\epsilon\Big(\frac{\Lambda_i}{2t_i}\Big)^{(n-2)p_0-2} &if\  p_0<\frac n{n-2}
  \end{cases}
+O(\epsilon^{1+\sigma}),
\end{split}
\end{align}
where $B_2$ and $\mathcal I_i$ are as in Lemma \ref{lema2} and Lemma\ref{lema2'}.
\end{Lem}

\begin{proof}
Recall that
\begin{align}\label{a5}
\begin{split}
  \tilde J_0&=\int_{\Omega}a(x)\nabla W_1\cdot\nabla W_2
-\frac1{p_0+1}\int_{\Omega}a(x)|W_2|^{p_0+1}dx-\frac1{q_0+1}\int_{\Omega}a(x)|W_1|^{q_0+1}dx\\
&=\int_{\Omega}a(x)W_1(-\Delta W_2)-\frac1{q_0+1}\int_{\Omega}a(x)|W_1|^{q_0+1}dx\\
&\quad-\frac1{p_0+1}\int_{\Omega}a(x)|W_2|^{p_0+1}dx-\int_{\Omega}\nabla a(x)\cdot\nabla W_2 W_1\\
&=\int_{\Omega}a(x)\sum_{j=1}^\kappa PU_j\sum_{i=1}^\kappa U_i^{q_0}-\frac1{q_0+1}\int_{\Omega}a(x)(\sum_{j=1}^\kappa PU_j)^{q_0+1}dx\\
&\quad-\frac1{p_0+1}\int_{\Omega}a(x)(\sum_{j=1}^\kappa PV_j)^{p_0+1}dx-\int_{\Omega}\nabla a(x)\cdot\nabla W_2 W_1.
\end{split}
\end{align}
By Lemma \ref{lema1}-Lemma \ref{lema3}, Lemma \ref{lemP'} and Lemma \ref{lemP'2} we can obtain \eqref{main}.

More precisely, firstly,
\begin{align}\label{e1}
\begin{split}
&\int_{\Omega}a(x)\sum_{j=1}^\kappa PU_j\sum_{i=1}^\kappa U_i^{q_0}\\
=&\int_{\Omega}a(x)\sum_{j=1}^\kappa PU_j U_j^{q_0}+\int_{\Omega}a(x)\sum_{i\neq j}^\kappa PU_jU_i^{q_0}\\
=&\int_{\Omega}a(x)\sum_{j=1}^\kappa U_j^{q_0+1}+\int_{\Omega}a(x)\sum_{j=1}^\kappa( PU_j-U_j) U_j^{q_0}
+\int_{\Omega}a(x)\sum_{i\neq j}^\kappa PU_jU_i^{q_0}.
%\\=&\int_{\Omega}a(x)\sum_{j=1}^\kappa U_j^{q_0+1}+\int_{\Omega}a(x)\sum_{j=1}^\kappa( PU_j-U_j) U_j^{q_0}+O(\epsilon^{1+\sigma}),
\end{split}
\end{align}
%where and in the following, $\sigma>0$ is a small constant.
On the other hand,
\iffalse
\begin{align*}
\begin{split}
&-\frac1{q_0+1}\int_{\Omega}a(x)(\sum_{j=1}^\kappa PU_j)^{q_0+1}dx\\
=&-\frac1{q_0+1}\int_{\Omega}a(x)\Big((\sum_{j=1}^\kappa PU_j)^{q_0+1}-(\sum_{j=1}^\kappa U_j)^{q_0+1}\Big)dx
-\frac1{q_0+1}\int_{\Omega}a(x)(\sum_{j=1}^\kappa  U_j)^{q_0+1}dx\\
=&-\int_{\Omega}a(x)\Big((\sum_{j=1}^\kappa PU_j)^{q_0}+(\sum_{j=1}^\kappa U_j)^{q_0}\Big)\sum_{j=1}^\kappa(PU_j-U_j)dx\\&
\quad-\frac1{q_0+1}\int_{\Omega}a(x)\sum_{j=1}^\kappa  U_j^{q_0+1}dx-\frac1{q_0+1}\int_{\Omega}a(x)\Big((\sum_{j=1}^\kappa  U_j)^{q_0+1}-\sum_{j=1}^\kappa  U_j^{q_0+1}\Big)dx\\
=&-\frac1{q_0+1}\int_{\Omega}a(x)\sum_{j=1}^\kappa  U_j^{q_0+1}dx+O(\epsilon^{1+\sigma}).
\end{split}
\end{align*}
\fi
\begin{align}\label{e2}
\begin{split}
&-\frac1{q_0+1}\int_{\Omega}a(x)(\sum_{j=1}^\kappa PU_j)^{q_0+1}dx\\
=&-\frac1{q_0+1}\sum_{i=1}^\kappa\int_{B_\eta(\xi_{i,\epsilon})}a(x)( PU_i)^{q_0+1}dx-\sum_{i=1}^\kappa\int_{B_\eta(\xi_{i,\epsilon})}a(x)\sum_{i\neq j}^\kappa PU_j(PU_i)^{q_0}+O(\epsilon^{1+\sigma})\\
=&-\frac1{q_0+1}\sum_{i=1}^\kappa\int_{B_\eta(\xi_{i,\epsilon})}a(x)U_i^{q_0+1}dx
-\frac1{q_0+1}\sum_{i=1}^\kappa\int_{B_\eta(\xi_{i,\epsilon})}a(x)(( PU_i)^{q_0+1}-U_i^{q_0+1})dx\\
&-\sum_{i=1}^\kappa\int_{B_\eta(\xi_{i,\epsilon})}a(x)\sum_{i\neq j}^\kappa PU_jU_i^{q_0}
-\sum_{i=1}^\kappa\int_{B_\eta(\xi_{i,\epsilon})}a(x)\sum_{i\neq j}^\kappa PU_j((PU_i)^{q_0}-U_i^{q_0})+O(\epsilon^{1+\sigma})\\
=&-\frac1{q_0+1}\sum_{i=1}^\kappa\int_{B_\eta(\xi_{i,\epsilon})}a(x)U_i^{q_0+1}dx
-\sum_{i=1}^\kappa\int_{\Omega}a(x)U_i^{q_0}(PU_i-U_i)dx\\&-\int_{\Omega}a(x)\sum_{i\neq j}^\kappa PU_jU_i^{q_0}
+O(\epsilon^{1+\sigma}).
\end{split}
\end{align}
In fact, in the above estimates, we prove in two different cases:

When $p_0<\frac n{n-2}$, it holds that for any small $\theta_0>0$,
\begin{align*}
\begin{split}
&\sum_{i=1}^\kappa\int_{B_\eta(\xi_{i,\epsilon})}a(x)\Big(\sum_{i\neq j}^\kappa PU_j\Big)^2(PU_i)^{q_0-1}\\&
=\begin{cases}\Big(\frac\delta\eta\Big)^{\frac{2n(p_0+1)}{q_0+1}-\theta_0}, &if\ \frac{n(p_0+1)(q_0-1)}{q_0+1}\geq n,\\
\Big(\frac\delta\eta\Big)^{p_0n}, &if\ \frac{n(p_0+1)(q_0-1)}{q_0+1}<n
\end{cases}=O(\epsilon^{1+\sigma})
\end{split}
\end{align*}
and similarly,
\begin{align*}
\begin{split}
&
-\sum_{i=1}^\kappa\int_{B_\eta(\xi_{i,\epsilon})}a(x)\sum_{i\neq j}^\kappa PU_j((PU_i)^{q_0}-U_i^{q_0})=O(\epsilon^{1+\sigma}).
\end{split}
\end{align*}

While when $p_0>\frac n{n-2}$, it is directly that
\begin{align*}
\begin{split}
&\sum_{i=1}^\kappa\int_{B_\eta(\xi_{i,\epsilon})}a(x)\Big(\sum_{i\neq j}^\kappa PU_j\Big)^2(PU_i)^{q_0-1}
=O\Big(\Big(\frac\delta\eta\Big)^{2(n-2)}\Big)
=O(\epsilon^{1+\sigma})
\end{split}
\end{align*}
and
\begin{align*}
\begin{split}
&
-\sum_{i=1}^\kappa\int_{B_\eta(\xi_{i,\epsilon})}a(x)\sum_{i\neq j}^\kappa PU_j((PU_i)^{q_0}-U_i^{q_0})=O\Big(\Big(\frac\delta\eta\Big)^{2(n-2)}\Big)=O(\epsilon^{1+\sigma}).
\end{split}
\end{align*}

Combining \eqref{e1} and \eqref{e2}, we obtain then
\begin{align}\label{a5'}
\begin{split}
  \tilde J_0&=-\frac1{q_0+1}\sum_{i=1}^\kappa\int_{B_\eta(\xi_{i,\epsilon})}a(x)U_i^{q_0+1}dx
+O(\epsilon^{1+\sigma})\\
&\quad-\frac1{p_0+1}\int_{\Omega}a(x)(\sum_{j=1}^\kappa PV_j)^{p_0+1}dx-\int_{\Omega}\nabla a(x)\cdot\nabla W_2 W_1.
\end{split}
\end{align}

Next, we have
\begin{align*}
\begin{split}
&-\frac1{p_0+1}\int_{\Omega}a(x)(\sum_{j=1}^\kappa PV_j)^{p_0+1}dx
\\=&-\frac1{p_0+1}\sum_{i=1}^\kappa\int_{B_\eta(\xi_{i,\epsilon})}a(x)V_i^{p_0+1}dx
-\sum_{i=1}^\kappa\int_{B_\eta(\xi_{i,\epsilon})}a(x)V_i^{p_0}(PV_i-V_i)dx
+O(\epsilon^{1+\sigma}).
\end{split}
\end{align*}
Finally, by Lemma \ref{lemP'} and Lemma \ref{lemP'2'}, we can get that
\begin{align*}
\begin{split}
&\int_{\Omega}\nabla a(x)\cdot\nabla W_2 W_1=O(\epsilon^{1+\sigma}).
\end{split}
\end{align*}

Substitute the above estimates into \eqref{a5} or \eqref{a5'}, in view of $A_1=B_1$, we have
\begin{align*}
\begin{split}
&I_0(W_1,W_2)\\
=& \frac{q_0 A_1}{q_0+1}\sum_{i=1}^\kappa\Big(a(\xi_{i})+\langle\nabla a(\xi_{i}),\gamma(\xi_{i})\rangle t_i\epsilon \Big)
-\frac{B_1}{p_0+1}\sum_{i=1}^\kappa\Big(a(\xi_{i})+\langle\nabla a(\xi_{i}),\gamma(\xi_{i})\rangle t_i\epsilon \Big)\\
&+\begin{cases} \frac{b_{n,p_0}B_2}{\gamma_n}\sum_{i=1}^\kappa a(\xi_{i})\epsilon\Big(\frac{\Lambda_i}{2t_i}\Big)^{n-2} &if p_0>\frac n{n-2}\\
  \frac{b_{n,p_0}\mathcal I_i}{\gamma_n}\sum_{i=1}^\kappa a(\xi_{i})\epsilon\Big(\frac{\Lambda_i}{2t_i}\Big)^{(n-2)p_0-2} &if p_0<\frac n{n-2}
  \end{cases}
+O(\epsilon^{1+\sigma})\\
=&\frac {2A_1}n\sum_{i=1}^\kappa\Big(a(\xi_{i})+\langle\nabla a(\xi_{i}),\gamma(\xi_{i})\rangle t_i\epsilon \Big)\\&+\begin{cases} \frac{b_{n,p_0}B_2}{\gamma_n}\sum_{i=1}^\kappa a(\xi_{i})\epsilon\Big(\frac{\Lambda_i}{2t_i}\Big)^{n-2} &if p_0>\frac n{n-2}\\
  \frac{b_{n,p_0}\mathcal I_i}{\gamma_n}\sum_{i=1}^\kappa a(\xi_{i})\epsilon\Big(\frac{\Lambda_i}{2t_i}\Big)^{(n-2)p_0-2} &if p_0<\frac n{n-2}
  \end{cases}
+O(\epsilon^{1+\sigma}),
\end{split}
\end{align*}
which gives \eqref{main}.

\end{proof}

\begin{Lem}\label{lema5}

There holds
\begin{align}\label{main'}
\begin{split}
I_\epsilon(W_1,W_2)=&\tilde J_0(\vec\xi,\vec\Lambda,\vec t)
-\epsilon\log\epsilon\frac{n(n-1)}{n-2}\Big(\frac{A_1}{(q_0+1)^2}+\frac{B_1}{(p_0+1)^2}\Big)\sum_{j=1}^\kappa a(\xi_{i})\\
&-\epsilon\Big(\frac{\beta A_1}{(q_0+1)^2}+\frac{\alpha B_1}{(p_0+1)^2}\Big)\sum_{j=1}^\kappa a(\xi_{i})\\
&-\epsilon\Big(\frac{nA_1}{(q_0+1)^2}+\frac{nB_1}{(p_0+1)^2}\Big)\sum_{j=1}^\kappa a(\xi_{i})\log \Lambda_i\\&
+\epsilon\Big(\frac{A_3}{q_0+1}+\frac{B_3}{p_0+1}\Big)\sum_{j=1}^\kappa a(\xi_{i}).
\end{split}
\end{align}

\end{Lem}
\begin{proof}

Using the elementary estimate that for $c\geq0,b\in\R$ and $r>0$
\begin{align*}
\frac{c^{r+1-b\epsilon}}{r+1-b\epsilon}-\frac{c^{r+1}}{r+1}=\epsilon\Big(\frac{c^{r+1}b}{(r+1)^2}-\frac{c^{r+1}b\log c}{r+1}\Big)+o(\epsilon),
\end{align*}
we have the expansion
\begin{align}\label{main''}
\begin{split}
I_\epsilon(W_1,W_2)&=\tilde J_0(\vec\xi,\vec\Lambda,\vec t)-\frac{\alpha\epsilon}{(p_0+1)^2}\int_\Omega a(x)W_2^{p_0+1}
-\frac{\beta\epsilon}{(q_0+1)^2}\int_\Omega  a(x)W_1^{q_0+1}\\&\quad
+\frac\epsilon{p_0+1}\int_\Omega  a(x)W_2^{p_0+1}\log W_2+\frac\epsilon{q_0+1}\int_\Omega  a(x)W_1^{q_0+1}\log W_1+o(\epsilon).
\end{split}
\end{align}

Firstly, as we have shown in Lemma \ref{lema4},
\begin{align}\label{a8}
\begin{split}
&-\frac{\alpha\epsilon}{(p_0+1)^2}\int_\Omega  a(x)W_2^{p_0+1}-\frac{\beta\epsilon}{(q_0+1)^2}\int_\Omega  a(x)W_1^{q_0+1}\\
=&-\epsilon\Big(\frac{\alpha B_1}{(p_0+1)^2}+\frac{\beta A_1}{(q_0+1)^2}\Big)\sum_{i=1}^\kappa a(\xi_{i})+o(\epsilon).
\end{split}\end{align}

Moreover,
\begin{align}\label{a9}
\begin{split}
&\frac\epsilon{q_0+1}\int_\Omega a(x) W_1^{q_0+1}\log W_1\\
=&\frac\epsilon{q_0+1}\sum_{j=1}^\kappa\int_{B_{\eta}(\xi_{j,\epsilon})} a(x) W_1^{q_0+1}\log W_1+o(\epsilon)\\
=&\sum_{j=1}^\kappa\Big(-\frac{\epsilon n}{(q_0+1)^2}\log\delta_j\int_{B_{\eta}(\xi_{j,\epsilon})} a(x) W_1^{q_0+1}\\&\qquad
+\frac\epsilon{q_0+1}\int_{B_{\eta}(\xi_{j,\epsilon})}  a(x) W_1^{q_0+1}\log \big(\delta_j^{\frac n{q_0+1}}U_j+\delta_j^{\frac n{q_0+1}}(W_1-PU_j)\big)
\Big)+o(\epsilon)\\
=&\frac{\epsilon}{q_0+1}\sum_{j=1}^\kappa \Big(-\frac{  A_1n}{q_0+1}\log\delta_ja(\xi_j)+a(\xi_j)A_3
\Big)+o(\epsilon).
\end{split}\end{align}
Similarly,
\begin{align}\label{a10}
\begin{split}
&\frac\epsilon{p_0+1}\int_\Omega a(x) W_2^{p_0+1}\log W_2
=\frac{\epsilon}{p_0+1}\sum_{j=1}^\kappa \Big(-\frac{B_1n}{p_0+1}\log\delta_ja(\xi_j)+a(\xi_j)B_3
\Big)+o(\epsilon).
\end{split}\end{align}

To sum up, from \eqref{main''}-\eqref{a10}, we obtain \eqref{main'}.

\end{proof}

\medskip
 \noindent\textbf{Acknowledgments}
Guo was supported by NSFC grants (No.12271539).
Peng was supported by NSFC grants (No.11831009).

\end{document}